\documentclass[12pt]{article}
\usepackage{amssymb}
\usepackage{amsmath,amsthm}
\usepackage{amsfonts}
\usepackage{graphicx}
\usepackage{graphics}
\usepackage{lineno}
\usepackage{xcolor}
\usepackage{ulem}

\numberwithin{equation}{section}
\numberwithin{footnote}{section}
\newtheorem{Theorem}{Theorem}

\newtheorem{cor}{Corollary}
\newtheorem{lemma}{Lemma}

\newcommand{\bp}{\begin{proof}}
\newcommand{\ep}{\end{proof}}
\newcommand{\be}{\begin{equation}}
\newcommand{\ee}{\end{equation}}
\newcommand{\bes}{\begin{equation*}}
\newcommand{\ees}{\end{equation*}}
\newcommand{\mbf}{\mathbf}
\newcommand{\bdx}{\mathbf{x}}

\newcommand{\bsy}{\boldsymbol}

\begin{document}

\date{\small\textsl{\today}}
\title{
Radial polynomials as alternatives to smooth radial basis functions  and their applications
}

\author{Fatemeh Pooladi 
\large , Hossein Hosseinzadeh$^*$   \\
\begin{footnote}{Corresponding author. \newline 
{\em  E-mail addresses:}
\newline Fpooladi13@mehr.pgu.ac.ir (F. Pooladi).
\newline h$_-$hosseinzadeh@aut.ac.ir \& hosseinzadeh@pgu.ac.ir (H. Hosseinzadeh),
}$\vspace{.2cm} $
\end{footnote}
\small{\em  Department of Mathematics, Persian Gulf University, Bushehr, Iran.}
\vspace{-1mm}} \maketitle

\begin{abstract}
Because of the high approximation power and simplicity of computation of smooth radial basis functions (RBFs), in recent decades they have received much attention for function approximation.  These RBFs contain a shape parameter that regulates the relation between their accuracy and stability. A difficulty in approximation via smooth RBFs is optimal selection of their shape parameter. The aim of this paper is to introduce an alternative for smooth RBFs, which, in addition to overcoming this difficulty, its approximation power is almost equal to RBFs. Motivated by Tylor expansion of smooth RBFs that contains only even power polynomials, a linear space of radial polynomials of degree $2n$, denoted by $\mathcal{H}_n$, is introduced and its basis functions are determined. 
Some sets of radial polynomial basis functions are introduced; one is more appropriate for theoretical analyzis and one is perfect for computational purposes.
Theoretical and numerical results presented in this paper show the new linear space leads to stable and accurate numerical results for interpolation scattered data and solving partial differential equations (PDEs).
\vspace{.5cm}
\end{abstract}


\section{Introduction}
Radial basis functions (RBFs) have gained the attention of researchers because they are able to solve partial differential equations (PDEs) and boundary value problems (BVPs) defined on computational domains with complicated geometries, easily. Also, RBFs are successfully applied for the classification of big data in high dimensional domains \cite{chen2017research, soleymani2023error, wu2021optimization}. 
Infinitely smooth RBFs may yield very accurate numerical results when the unknown function is sufficiently smooth \cite{wendland2004scattered}. The accuracy of infinitely smooth RBFs mainly depends on two factors: Kernel function $\phi$ and shape parameters, $\epsilon$. In RBF approximation, the shape parameter should be determined by the user, and the highest accuracy is often obtained at small shape parameters, which may yield unstable results \cite{mongillo2011choosing}.
A significant number of researchers have focused on the enhancement of stability of flat RBFs, theoretically. Fornberg et al. \cite{fornberg2004stable} used the contour-Pad$\acute{e}$ algorithm to obtain the RBF interpolating function and Gonzalez et al. \cite{gonzalez2015laurent} treat singularity of the RBF matrix and compute the inverse of the RBF matrix by Laurent expansion. Both studies found that the interpolating function obtained by infinitely smooth RBFs converge to polynomial functions when $\epsilon$ tends to zero. Moreover, Driscoll and Fornberg \cite{driscoll2002interpolation} show interpolation function converges to a polynomial function that interpolates the same data when $\epsilon$ tends to zero. 
Numerical studies by Boyd and Alfaro \cite{boyd2013constructing, boyd2013hermite, boyd2010error} show that in one dimensional, the cardinal functions of Gaussian RBFs are approximately equal to the product of some functions to cardinal functions of Lagrangian polynomials. Then linear spaces of Gaussian RBFs and polynomials are very close to each other in this situation.

Adding polynomials to radial basis functions is also one approach to enhance the approximation power of RBFs. For instance an orthogonal augment radial basis function (OA-RBF) method is proposed in \cite{wu2019global} for estimating a general class of Sobol indices. Also the combination of polyharmonic splines (PHS) and polynomials with high degree is applied in \cite{bayona2019insight, bayona2019role} as a powerful and robust numerical approach for local interpolation. Polyharmonic splines are used as radial basis function with augmented polynomials in \cite{orucc2022strong} to solve coupled damped Schrödinger system, numerically. In \cite{shahane2021high} and \cite{bartwal2022application} PHS with appended polynomials is applied for the solution of incompressible Navier-Stokes and heat conduction equations, respectively. Besides, there are several researches which highlight appended polynomials to RBFs give rapid convergence \cite{ cao2022polynomial, lamichhane2023localized, shahane2021high, oruc2022radial}.



Taylor expansion of infinitely smooth RBFs contain only even power of radius, $r$. This fact motivate us to search for a space of polynomials containing just even powers to approximate the RBFs. So a linear space of polynomials of degree less than or equal to $2n$, denoted by  $\mathcal{H}_n$, is proposed, where by increasing the number of its basis, distance from the RBFs to it vanishes, exponentially.
Dimension of the new linear space is determined and its basis functions are obtained, analytically. Then the distance from smooth RBFs to $\mathcal{H}_n$ can be calculated, numerically. Numerical results show that although  $\mathcal{H}_n \subset \mathcal{P}_{2n}$ but the distances from the RBFs to the both linear spaces are nearly equal. Then $\mathcal{H}_n$ may be supposed as the smallest subspace of $\mathcal{P}_{2n}$ where smooth RBFs tend to it as fast as possible. Since usual basis functions of $\mathcal{H}_n$ lead unstable numerical results, some new basis functions also are proposed in this paper for the linear space where enhance the stability and accuracy of the numerical results, significantly. 

The rest of the paper is organized as follows. The relation between the RBFs and polynomials is studied in Section \ref{sec2}. Then, in Section \ref{sec3} the new linear space of radial polynomials is introduced and its dimension and basis functions are obtained. The obtained basis functions are replaced by more stable ones in Section \ref{sec4} to enhance robustness of the linear space. Some numerical studies are presented in Section \ref{sec5} to show ability of the new linear space versus smooth RBFs. Finally, the paper is completed by Section \ref{sec6} including some conclusions and final remarks.

\section{Radial basis functions }\label{sec2}
Using RBFs in scattered data approximation was proposed by Hardy \cite{hardy1971multiquadric}. 
From a point of review, RBFs can be divided into three groups, infinitely smooth, piecewise smooth and compact support RBFs. Some well known RBFs are listed in Table \ref{tab1}.
\begin{table}[ht]
  \begin{center}
    \caption{Some important RBFs presented in the literature. }\label{tab1}
    \begin{tabular}{lll} 
      \hline
      \hline
      RBF & name & function \\
      \hline
       infinitely smooth  & Gaussian (GA)  & $\exp(-\epsilon^2 r^2)$      \\
         & Multiquadric (MQ) & $\sqrt{1+\epsilon^2 r^2}$  \\
         & Inverse Multiquadric (IMQ) & $1/\sqrt{1+\epsilon^2 r^2}$  \\
         & Inverse Quadric (IQ) & $1/(1+\epsilon^2 r^2)$  \\
      \hline
      piecewise smooth & Thin plate spline (TPS)  &  $r^{2n} \ln(r)$   \\
       & Poly-harmonic Spline (PHS) & $r^{2n+1}$ \\
      \hline
      compactly support & Wendland $C^0$ & $(1-r)_+^2$              \\
        & Wendland $C^2$ & $(1-r)_+^4 (1+4r)$  \\
      \hline
      \hline
    \end{tabular}
  \end{center}
\end{table}

Suppose given data of the form $(\mbf{x}_i , b_i)$ for $i = 1, 2, . . . , N$ where $\mbf{x}_i\in\Omega\subseteq\mathbb{R}^d$ and $u_i\in\mathbb{R}$. We seek for an interpolant $\bar{u}(\mbf{x})$ which satisfies
\begin{align}\label{eq2.1}
\bar{u}(\mbf{x}_i)=u_i  ~, ~~~ i=1, 2, ..., N.
\end{align}
In RBF interpolation approach, $\bar{u}(\mbf{x})$ is a linear function of radial basis functions $\phi(r_1(\mbf{x}))$, $\phi(r_2(\mbf{x}))$, $...,\phi(r_N(\mbf{x}))$ as
\begin{equation}\label{eq2.2}
\bar{u}(\bdx) = \sum_{i=1}^N  \lambda_i \phi(r_i(\mbf{x})) ,
\end{equation}
such that $\lambda_i$ are constants and 
$
r_i(\mbf{x})=\| \mbf{x} - \mbf{x}_i \| .
$
In fact, interpolant $\bar{u}$ belongs to the space 
\be\label{S}
S=\langle \phi(r_1(\mbf{x})) , \phi(r_2(\mbf{x})) , ..., \phi(r_N(\mbf{x})) \rangle . 
\ee
The value of RBF $\phi(r_i(\mbf{x}))$ depends on kind of the RBFs that used, the distance from the input data $\mbf{x}$ to the collocation point $\mbf{x}_i$ and shape parameter, $\epsilon$. Equation (\ref{eq2.1}) can be stated as a system of linear equations
\begin{equation}\label{eq2.4}
\boldsymbol{\Phi} \boldsymbol{\lambda} = \mathbf{b} .
\end{equation}
This system is non-singular when positive definite RBFs (such as Gaussian RBF) with constant shape parameters are used \cite{fasshauer2007meshfree, wendland2004scattered}. It has been shown that the interpolant function admits spectral convergence when infinitely smooth RBFs are applied \cite{wendland2004scattered}.
However the accuracy and the stability of these RBFs depend on the number of data points and the value of $\epsilon$  \cite{wendland2004scattered, madych1992miscellaneous, schaback1995error}. It is well-known that interpolation function $\bar{u}$ converges to a polynomial function, exponentially, when $\epsilon$ is sufficiently small or $N$ is sufficiently large \cite{driscoll2002interpolation, fornberg2007runge}.

\section{Proposed linear space of radial polynomials}\label{sec3}
This section introduces the new linear space of radial polynomials to be used as an alternative for linear space $S$ introduced in Equation (\ref{S}). The new linear space is free of shape parameter and it has high power of approximation.

\subsection{New linear space}
As is mentioned in equations (\ref{eq2.2}) and (\ref{S}) interpolation function $\bar{u}$ belongs to the linear space produced by RBFs, $S$.  Considering Taylor expansion of a smooth RBF, $\phi(r_i)$, we have
\begin{align}\label{eq3.1}
\phi(r_i) \simeq f_n(r_i) = 1+c_2\epsilon^2 r_i^2 + c_4 \epsilon^4 r_i^4  + \cdots + c_{2n} \epsilon^{2n} r_i^{2n} .
\end{align} 
Now instead of $S$, we consider the space generated by polynomials $f_n(r_i)$, as follows
\begin{align}\label{eq.Hh}
{\mathcal{H}}=\langle f_n(r_1), f_n(r_2), ..., f_n(r_N)\rangle .
\end{align}
One can approximate the unknown function $\bar{u}$ in $\mathcal{H}$ instead of $S$, but, because of complexity of computation, the space $\mathcal{H}_n$ is proposed as follows: 
\be\label{Hn}
\mathcal{H}_{n}=\langle r_1^{2n}, r_2^{2n}, ... , r_N^{2n}\rangle .
\ee
Using $\mathcal{H}_n$ for approximation has less computational cost than $\mathcal{H}$ and in continue it is proved that $\mathcal{H} \subseteq \mathcal{H}_n$.

\subsection{Relation between $\mathcal{H}_n$ and $\mathcal{H}$}
From equations (\ref{eq3.1}) and (\ref{Hn}) we find $\mathcal{H}$ and $\mathcal{H}_{n}$ both include polynomials of degree less than or equal to $2n$. In this subsection, it is proved by a theorem that, $\mathcal{H}$ is embedded in $\mathcal{H}_n$ if $ \mathcal{H}_{n}$ is full rank. 
Note that $\mathcal{H}_{n}$ is full rank if $r_0^{2n}\in\mathcal{H}_{n}$ for each $\bdx_0\in\Omega$. 

\begin{lemma}\label{le.dim2}
If $\mathcal{H}_n$ is full rank then $\mathcal{H}_{n-1}$ is also full rank for $n\geq 1$.
\end{lemma}
\begin{proof}
Suppose $\mathcal{H}_n$ is full rank, so
for any computational point $\mathbf{x}_0\in\mathbb{R}^d$ there are coefficients $\lambda_1, ...,\lambda_N$ in $\mathbb{R}$ such that
\begin{align}\label{eq.sum1}
r_0^{2n} = \sum_{i=1}^N \lambda_i \, r_i^{2n} .
\end{align}
Now imposing Laplace operator, $\nabla^2$, on both sides of Equation (\ref{eq.sum1}) results in 
\begin{align}\label{eq.sum2}
r_0^{2n-2} = \sum_{i=1}^N \lambda_i \, r_i^{2n-2} ,
\end{align}
which shows that $\mathcal{H}_{n-1}$ is also full rank.
\end{proof}

\begin{cor} \label{co.dim} 
If $\mathcal{H}_n$ is full rank then $\mathcal{H}_{j}$ also is full rank for each $j\leq n$.
\end{cor}
\begin{proof}
It can be proved by Lemma \ref{le.dim2} and induction on $j$ when $j$ decreases from $n$ to $1$ step by step.
\end{proof}

\begin{lemma}\label{le.dim22}
If $\mathcal{H}_n$ is full rank then $\mathcal{H}_{n-1}\subseteq \mathcal{H}_n$ for $n\geq 1$.
\end{lemma}
\begin{proof}
Let $\mathbf{x}_0$ is a computational point in $\mathbb{R}^d$, and $\nabla^2$ is Laplace operator. Therefore
\begin{align}\label{eq3.10}
r_0^{2n-2} = \frac{1}{4n^2} \nabla^2 r_0^{2n} 
= \frac{1}{4n^2}  \lim_{h \rightarrow 0}\frac{1}{h^2} \left( \sum_{j=0:2d} \lambda_j \, r_0^{2n} (x + h \mathbf{v}_j) \right) ,
\end{align}
where $\lambda_j$ and $v_j$ are appropriate finite difference coefficients and vectors, respectively  for $j=0,1, 2, ...,2d$. Since polynomials $r_0^{2n} (x_0 + h \mathbf{v}_j)$ are belong to $\mathcal{H}_n$ and $\mathcal{H}_n$ is closed space (because it has finite dimension), then the limit presented in right hand side of Equation (\ref{eq3.10}) also belongs to $\mathcal{H}_n$ and consequently $r_0^{2n-2}\in \mathcal{H}_n$. So, $\mathcal{H}_{n-1} \subseteq \mathcal{H}_n$.
\end{proof}

\begin{cor}\label{co.sub}
If $\mathcal{H}_n$ is full rank, then $\mathcal{H}_j \subseteq \mathcal{H}_n$ for $j\leq n$.
\end{cor}
\begin{proof}
It can be proved by Lemma \ref{le.dim22} and induction on $j$ when it decreases from $n$ to $1$ step by step.
\end{proof}

Now the important theorem can be achieved from the last corollary.
\begin{Theorem} \label{th1}
If $\mathcal{H}_n$ is full rank, then ${\mathcal{H}} \subseteq \mathcal{H}_{n}$.
\end{Theorem}
\begin{proof}
From Corollary \ref{co.sub} we have 
$\mathcal{H}_0 \subseteq \mathcal{H}_1 \subseteq ... \subseteq \mathcal{H}_{n-1} \subseteq \mathcal{H}_n$ where it leads
\begin{align*}
\mathcal{H}_{0} + \mathcal{H}_{1} + ... + \mathcal{H}_{n} = \mathcal{H}_{n} ,
\end{align*}
and consequently $f_n(r_i)\in\mathcal{H}_{n}$  for $i=1, 2, ..., N$, and  ${\mathcal{H}} \subseteq \mathcal{H}_{n}$.
\end{proof}

\subsection{Basis of $\mathcal{H}_n$}
Now we are going to find the basis functions of $\mathcal{H}_n$. From theorem \ref{th1} we have $\mathcal{H} \subseteq \mathcal{H}_{n}$ and consequently basis of $\mathcal{H}$ is a subset of basis of $\mathcal{H}_n$.
Note that elements of $\mathcal{H}_{n}$ are radial polynomials of degree $2n$ and its dimension is bounded by the dimension of polynomials of degree less than or equal to $2n$, denoted by $\mathcal{P}_{2n}$. We will see that the dimension of $\mathcal{H}_n$ is significantly smaller than the dimension of $\mathcal{P}_{2n}$ for $d \geq 2$. 
\begin{lemma}\label{le.f12}
If $r^2=\| \bdx \|^2$ and $m \leq n$, then dimension of linear space ${\langle{  r^2 \mathcal{P}_{m-1} - \mathcal{P}_{m}}\rangle}$ is equal to 
\begin{align}\label{md}
\begin{pmatrix}
m+d-2 \\
m-1
\end{pmatrix}
= \frac{(m+d-2)!}{(m-1) ! (d-1)!} .
\end{align}
\end{lemma} 
\begin{proof}
Since $r^2 \in \mathcal{P}_{2}$ then $ r^2 \mathcal{P}_{m-1}$ is a set of polynomials of degree less than or equal to $m+1$ which have factor $r^2$ in their components. Subtracting $\mathcal{P}_{m}$ from this set, only polynomials of degree $m-1$ remain which are multiplied by $r^2$, i.e.  
\begin{align*}
\langle r^2 \mathcal{P}_{m-1} - \mathcal{P}_m\rangle =  r^2 \langle\mathcal{P}_{m-1} - \mathcal{P}_{m-2}\rangle .
\end{align*} 
Polynomial basis functions of $\langle\mathcal{P}_{m-1} - \mathcal{P}_{m-2}\rangle$ satisfy condition
\begin{align*}
\prod_{k=1:d} x_k^{\alpha_k} ~~~s.t.~~ \sum_{k=1:d} \alpha_k = m-1 ,
\end{align*} 
where $x_k$ is $k_-$th component of $\bdx$ for $k=1, 2, ...,d$. The number of these basis functions can be calculated via Equation (\ref{md}).
\end{proof}
\begin{Theorem}\label{le.dim}
Dimension of $\mathcal{H}_{n}$ is less than or equal to
\begin{align*}
h(n)= 2 \sum_{i=0:n-1} 
\begin{pmatrix}
i+d-1 \\
i
\end{pmatrix}
+
\begin{pmatrix}
n+d-1 \\
n
\end{pmatrix} ,
\end{align*}
for $n\geq 2$.
\end{Theorem} 
\begin{proof}
The proof can be done via induction. For $n=0$ we have $\mathcal{H}_{0}=\{ 1 \}$ and $dim(\mathcal{H}_{0})=1$.
For $n=1$ if $\bdx_0$ is a computational point in $\in\mathbb{R}^d$ then 
\begin{align}\label{eq3.3}
r_0^2  & =\| \bdx - \bdx_0 \|^2 = (\bdx - \bdx_0)(\bdx - \bdx_0)^T= 
\| \bdx_0 \|^2 - 2 \bdx_0 \bdx^T + \| \bdx \|^2 \nonumber \\
           & = c_0 + \sum_{k=1:d} c_k x_k +r^2 ,
\end{align}
where $c_0, c_1, ..., c_d$ are constant numbers depending on $\bdx_0$ and $r^2=\| \bdx \|^2$. Then, $\mathcal{H}_{1}$ has one base of degree zero, $d$ basis of degree one and one base of degree two. 
For $n \geq 2$, by Equation (\ref{eq3.3}) we have
\begin{align*}
r_0^{2n}=(r_0^2)^n=(c_0 + \sum_{k=1:d} c_k x_k +r^2)^n = \sum_{\sum \alpha_k + \beta \leq n}  c(\alpha_1, \alpha_2, ..., \alpha_d, \beta) \prod_{k=1:d} x_k^{\alpha_k} r^{2 \beta}  ,
\end{align*}
where $c(\alpha_1, \alpha_2, ..., \alpha_d, \beta)$ are constant numbers depending on $\alpha_1, \alpha_2, ..., \alpha_d$ and $\beta$. The last equality shows the basis functions of $\mathcal{H}_{n}$ are in the form 
\begin{align*}
\prod_{k=1:d} x_k^{\alpha_k} r^{2 \beta} ~~~s.t.~~ \sum \alpha_k + \beta \leq n .
\end{align*}
To count basis functions of $\mathcal{H}_{n}$, we split $\mathcal{H}_{n}$ into disjoint subsets and count their basis functions separately by using Lemma \ref{le.f12}. These disjoint subsets are $\mathcal{P}_n, \langle\mathcal{P}_{n-1}r^2 - \mathcal{P}_n\rangle, \langle\mathcal{P}_{n-2}r^4 - \mathcal{P}_{n-1} r^2\rangle,$ $ ...,  \langle\mathcal{P}_{0}r^{2n} - \mathcal{P}_{1} r^{2n-1}\rangle$ which they lead
\begin{align} \label{al.1}
\mathcal{H}_{n}
= \sum_{i=0}^n \mathcal{P}_{n-i} r^{ 2 i }
&= \mathcal{P}_n + \sum_{i=1:n} \langle\mathcal{P}_{n-i} r^{2 i} - \mathcal{P}_{n-i-1} r^{2 (i-1)}\rangle \nonumber\\
& = \mathcal{P}_n + \sum_{i=1:n} \langle\mathcal{P}_{n-i} r^2-\mathcal{P}_{n-i-1}\rangle r^{2 (i-1)} ,
\end{align}
and by Lemma \ref{le.f12} we have
\begin{align} \label{al.2}
dim(\mathcal{H}_{n})
&= dim( \mathcal{P}_n) + \sum_{i=1:n} dim( \langle\mathcal{P}_{n-i} - \mathcal{P}_{n-i-1}\rangle ) \nonumber \\
& = \sum_{i=0:n} 
\begin{pmatrix}
i+d-1 \\
i
\end{pmatrix}
+ \sum_{i=1:n} 
\begin{pmatrix}
n-i+d-1 \\
n-i
\end{pmatrix} \nonumber \\
& = 2 \sum_{i=0:n-1} 
\begin{pmatrix}
i+d-1 \\
i
\end{pmatrix}
+ 
\begin{pmatrix}
n+d-1 \\
n
\end{pmatrix} .
\end{align}
\end{proof}

Theorem \ref{le.dim} presents a constrictive algorithm to obtain the basis of $\mathcal{H}_{n}$. Figure \ref{fig1} shows these basis functions for $d=1, 2$ and $3$ where in this figure $r^2=\| \mbf{x}\|^2$. From Theorem \ref{le.dim} dimension of $\mathcal{H}_{n}$ is equal to
\begin{align}\label{Dn}
 h(n) 
= \left\{ \begin{array}{ll}
        2n+1 & \mbox{if $d=1$} ,\\
        (n+1)^2 & \mbox{if $d=2$} ,\\
        (n+1)(n+2)(2n+3)/6 & \mbox{if $d=3$} , 
        \end{array} \right. 
\end{align} 
when it is full rank. Also from Theorem \ref{th1} dimension of linear space $\mathcal{H}$ is less than or equal to $h(n)$. 
Note that by Theorem \ref{le.dim} linear space $\mathcal{H}_n$ has at most $h(n)$ basis functions and it can be presented as
\begin{align}\label{eq.Hh}
\mathcal{H}_n =\langle r_1^{2n}, r_2^{2n}, ... , r_{h(n)}^{2n} \rangle ,
\end{align} 
when $\mbf{x}_1, \mbf{x}_2, ..., \mbf{x}_{h(n)}$ are different points in $\mathbb{R}^d$.
Some important relations between linear spaces $\mathcal{H}_{n}$ and $\mathcal{P}_{n}$ are obtained from Theorem \ref{le.dim} which are noted in the following corollary.

\begin{figure}\label{fig1} 
\centering 
\includegraphics[scale=0.4]{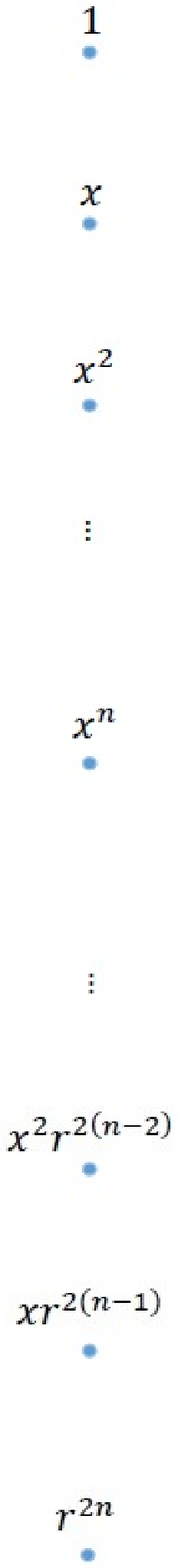} ~~~~~~~
\includegraphics[scale=0.4]{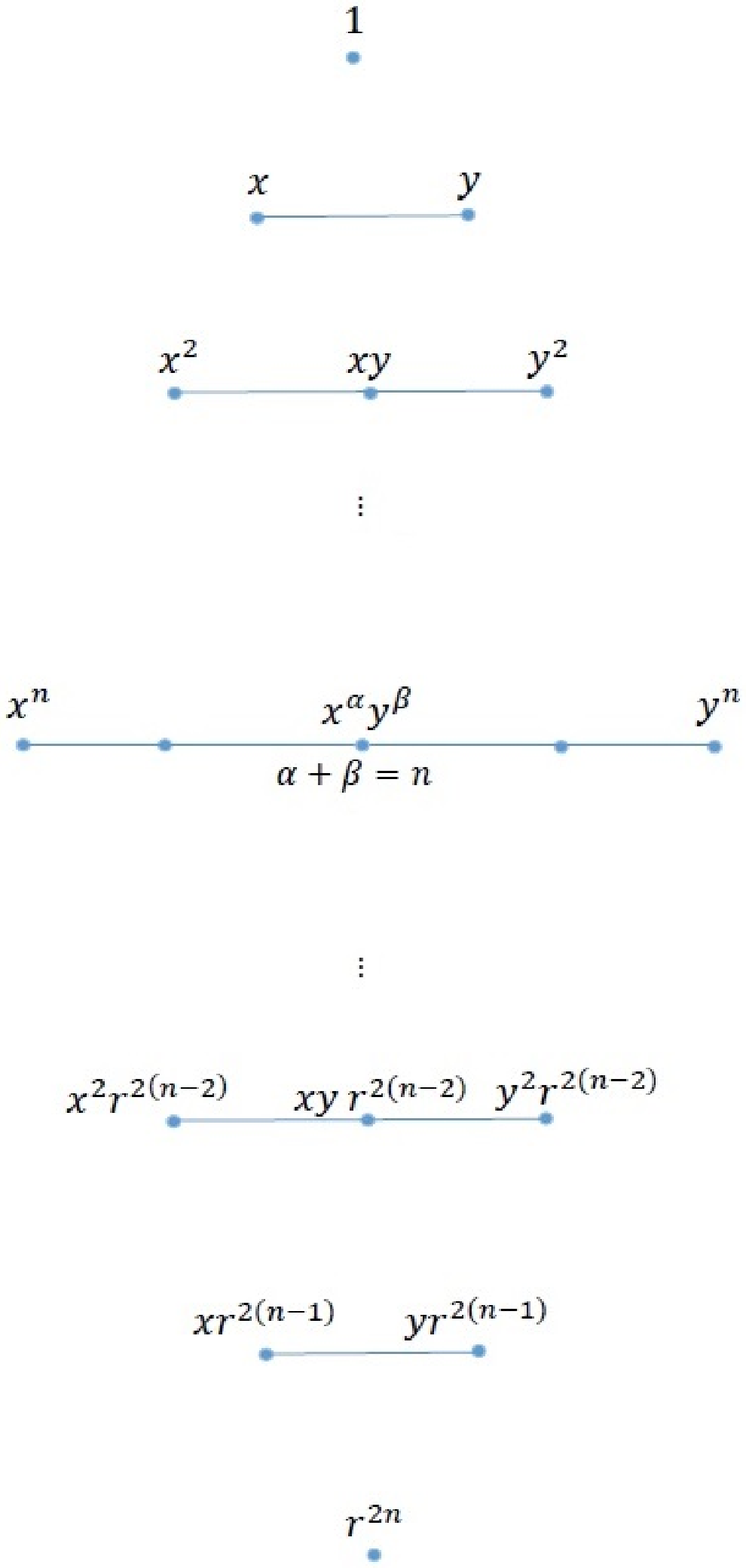} ~~~
\includegraphics[scale=0.4]{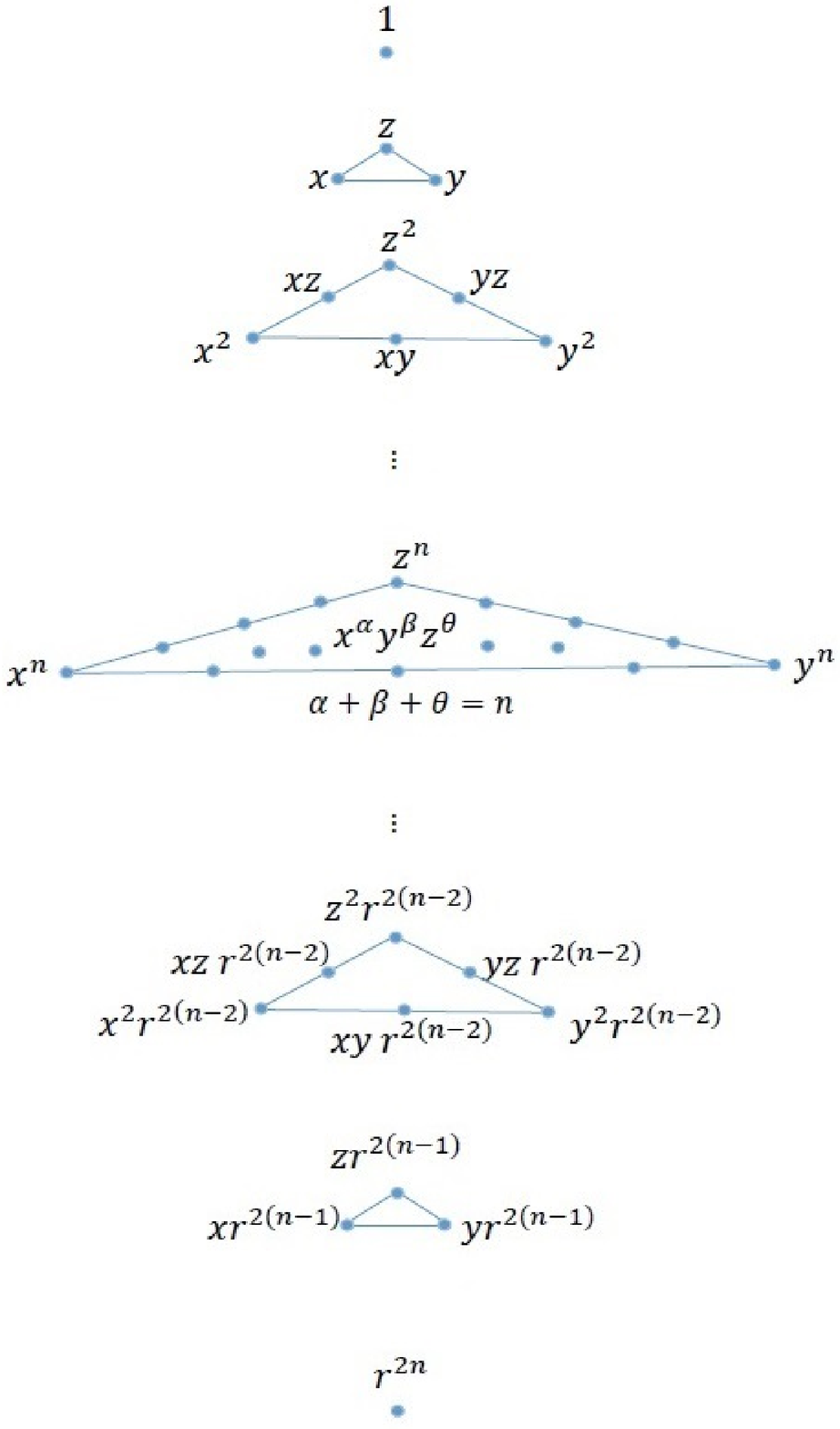} 
\caption{Diagram of basis functions of $\mathcal{H}_{n}$ for $d=1$ (left), $d=2$ (center) and $d=3$ (right). 
} 
\label{fig1} 
\end{figure} 

\begin{cor}\label{co.pol2}
Let $\mathcal{H}_{n}$ is linear space of radial polynomials of degree $2n$ and $\mathcal{P}_{n}$ is linear space of polynomials of degree less than or equal to $n$. If $\mathcal{H}_{n}$ is full rank then the following items are valid
\begin{itemize}
\item
$ \mathcal{P}_{n}  \subset \mathcal{H}_n \subseteq \mathcal{P}_{2n} ~,$
\item
$ dim( \mathcal{H}_n) = \, dim(\mathcal{P}_{n}) + \, dim(\mathcal{P}_{n-1}) ~,$
\item
$ \mathcal{H}_n = \mathcal{P}_{2n} ~~ \hbox{for} ~~ d=1 ~, $
\item
$ \mathcal{P}_{n+1} \not\subseteq \mathcal{H}_n \subset \mathcal{P}_{2n} ~~ \hbox{for} ~~ d \geq2 .$
\end{itemize}
\end{cor}

\section{Computationally efficient basis functions for $\mathcal{H}_n$}\label{sec4}
Basis functions $p_i= r_i^{2n}$, presented in Eq. (\ref{eq.Hh}), have a simple form but they are not computationally efficient, i.e. despite their simplicity, they are not numerically stable. Generally, to enhance the stability, the basis functions can be normalized by mapping the distance function $r_i$ to $[0,1]$ as 
\[
r_i=r_i/\max_{\mbf{x}\in\Omega}\{r_i\} . 
\]
By this normalization, the basis functions are mapped to $[0,1]$ and this enhances the stability. Note that the normalized basis functions, $p_i$, vanish close to the center point, $\mbf{x}_i$, and increase rapidly close to some boundary points far from the center point, specially when $n$ is large. For instance, $p_i \leq 10^{-12}$ for $r_i<0.5$ and $p_i=1$ at $r_i=1$ when $n\geq 20$. In this situation, two basis functions $p_i$ and $p_j$ are strongly dependent when there is a point $\mbf{y}\in\Omega$ satisfying
\[
\arg\max_{\mbf{x}\in\Omega} \{ \| \mbf{x}_i - \mbf{x} \| \}=\mbf{y}=\arg\max_{\mbf{x}\in\Omega} \{ \| \mbf{x}_j - \mbf{x} \| \}.
\] 
To check this fact, contour lines of elements of matrix $\mbf{A}=\boldsymbol{\Phi}^T\boldsymbol{\Phi}$ is shown in Figure \ref{fig25} (left) when $\mbf{x}_i=(i-1)/2n$ for $i=1, 2, ..., 2n+1$ and $n=15$. Since
\[
\mbf{A}[i,j]=\sum_{k=1:N} \phi(r_i(\mbf{x}_k))\phi(r_j(\mbf{x}_k)) \simeq N \int_{\Omega} \phi(r_i(\mbf{x}))\phi(r_j(\mbf{x})) \, d\mbf{x} ,
\]
matrix $\mbf{A}$ contains internal product between the basis functions. Therefore $\mbf{A}[i,j]$ shows dependency between $p_i$ and $p_j$. 
Thanks to Figure \ref{fig25} (left), one can see $\mbf{A}[i,j] \simeq 1$ for $1\leq i,j \leq n$ and $n+2\leq i,j \leq 2n+1$ where it shows high dependency between the basis functions. 
This dependency leads unstable numerical results and high condition number, defined as 
\bes
cond( \boldsymbol{\Phi} ) =\| \boldsymbol{\Phi} \| \| \boldsymbol{\Phi}^{-1} \| .
\ees
In the following subsections, two approaches are proposed to reduce the dependency and enrich the stability of numerical results by introducing some new basis functions for $\mathcal{H}_n$. 

\begin{figure}[ht]
\centering 
\includegraphics[scale=0.5]{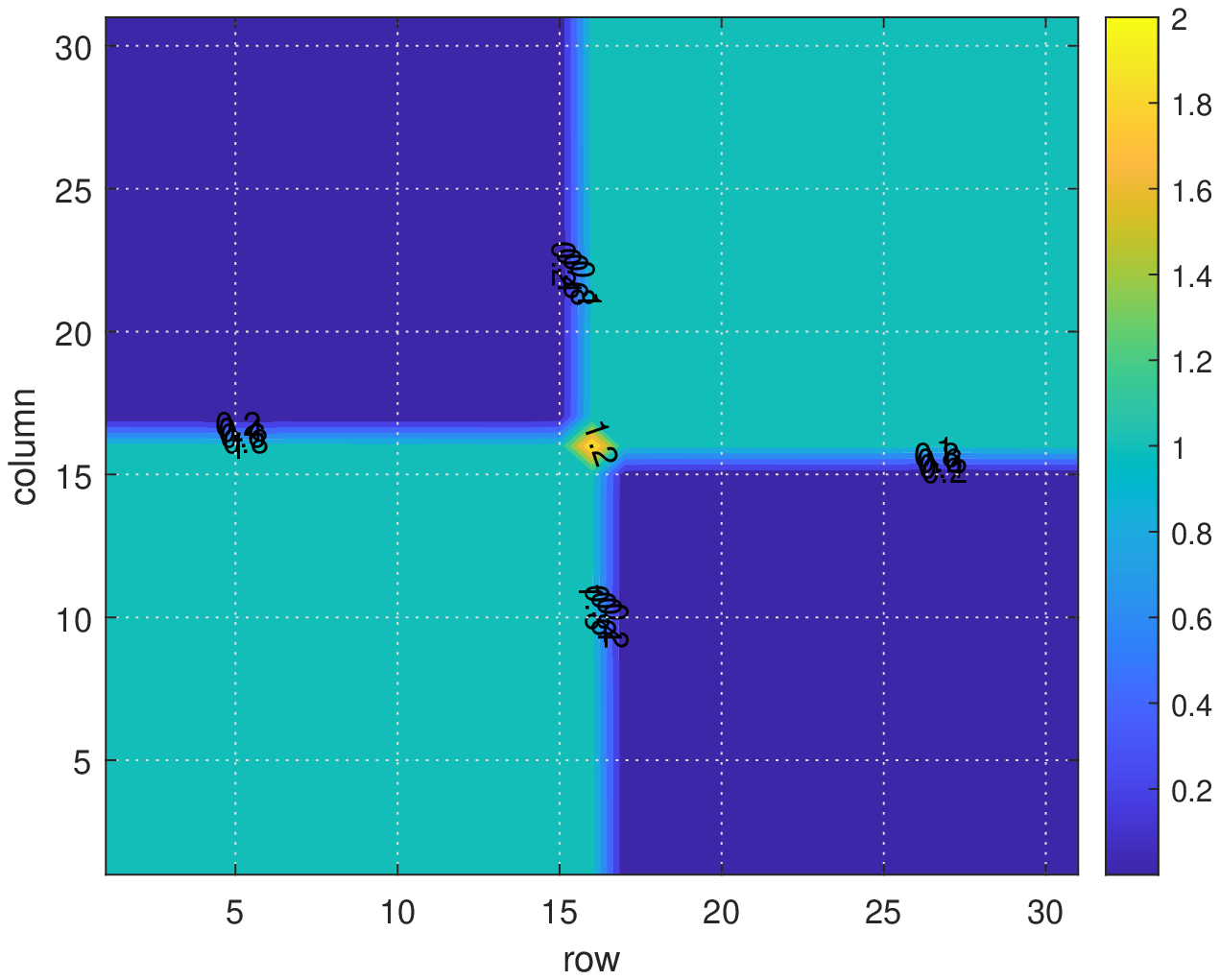}
\includegraphics[scale=0.5]{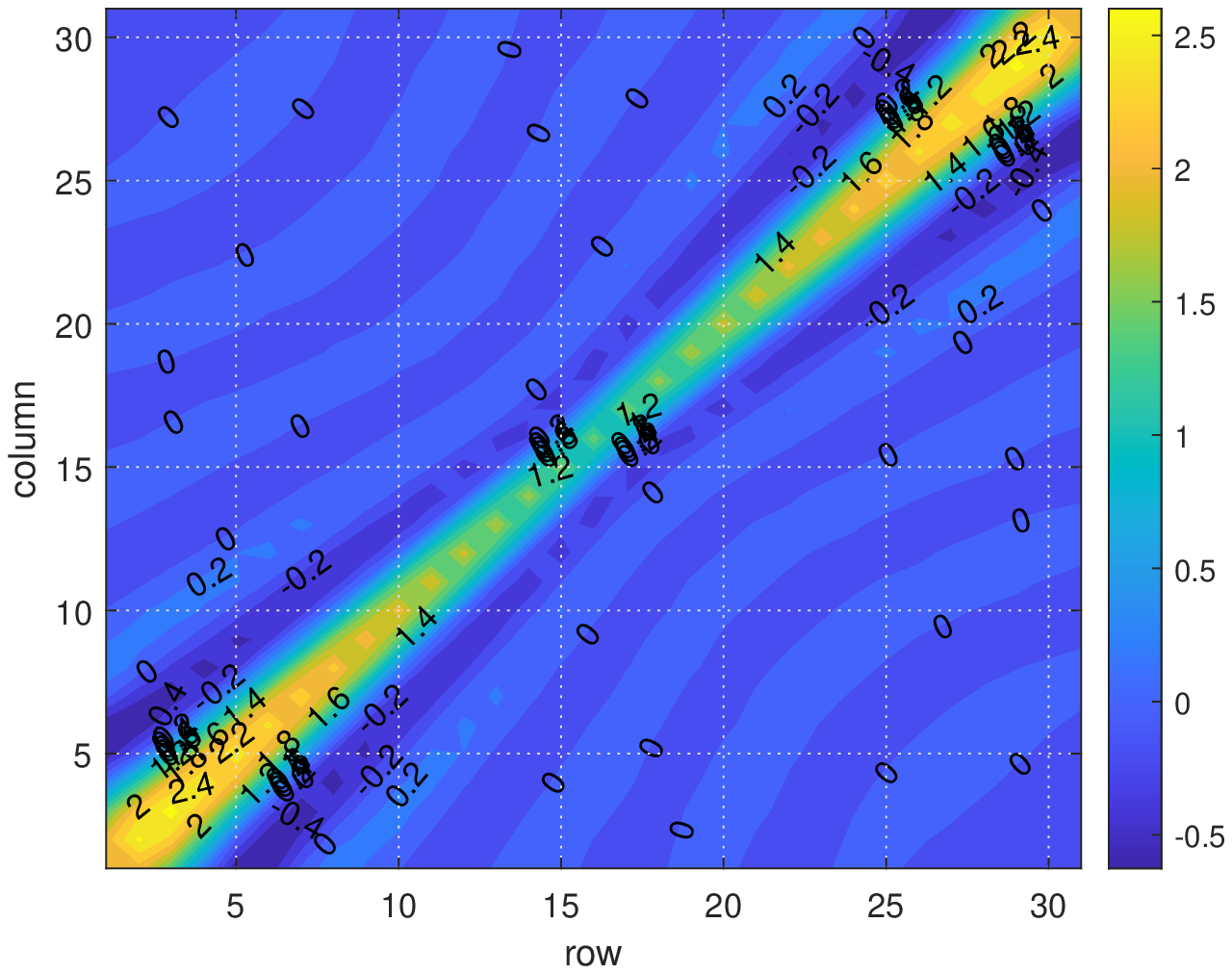} 
\caption{Contour lines of elements of matrix $\mbf{A}$ for polynomial basis functions $p_i$ (left) and $p_{i,2}$ (right). 
} 
\label{fig25} 
\end{figure} 

\subsection{Semi-cardinal basis functions}\label{sec4_1}
In order to find more stable basis functions, we focus on basis functions which are similar to the cardinal function, taking $1$ at a center point and $0$ at the others. This property enhances the stability by resulting a semi-diagonal coefficient matrix $\boldsymbol{\Phi}$. For this goal, we introduce basis functions
\[
p_{i,\boldsymbol{\alpha}}=\prod_{k=1:n} (r_i^2 - \alpha_k^2) ,
\]
for $i=1, 2, ..., h(n)$ and some  positive vector $\boldsymbol{\alpha}=[\alpha_1^2, \alpha_2^2, ..., \alpha_n^2]$. Basis function, $p_{i,\boldsymbol{\alpha}}$ is a polynomial function with roots $\alpha_1, \alpha_2, ..., \alpha_n$. Note that, $p_i=p_{i,\mbf{0}}$ where $\mbf{0}$ is the zero vector. A cross-section of the basis function is plotted in Figure \ref{fig2} when the center point is $0$ and $\alpha_k^2=t_k^\tau$ for $\tau=0,1,2,2.1$ where $t_k$ is $k $th positive root of Chebyshev polynomial of first kind of degree $2n+1$, i.e. 
\[
t_k=\cos\big(\frac{2k-1}{4n+2} \pi \big) .
\]
It is easy to find from Figure \ref{fig2} that $\tau>2$ does not yield flat function for the points far from the center, then we restrict our research on $\tau\leq 2$. 
As a result, three kinds of functions 
\begin{align}
p_{i,\tau} = \prod_{k=1:n} (r_i^2 - t_k^\tau) ~;~~~ \tau = 0, 1, 2 ,
\end{align} 
are considered here as semi-cardinal basis of $\mathcal{H}_n$. To test the stability of the new basis functions, condition number of the coefficient matrix, $\boldsymbol{\Phi}$, is calculated and is presented in Figure \ref{fig3}. For this figure center points are $\mbf{x}_i=(i-1)/2n$ for $i=1, 2, ..., 2n+1$ where $n=15$.  This figure shows that the condition number for $p_{i,2}$ is smaller than that of the others, specially Gaussian RBF with shape parameter $\epsilon = 1$. So, to enhance the numerical stability one can use $p_{i,2}$ as basis of $\mathcal{H}_n$ for numerical interpolation. To check dependency of basis functions $p_{i,2}$ and  $p_{j,2}$ for $i,j=1, 2, ..., N$, contour lines of elements of matrix $\mbf{A}=\boldsymbol{\Phi}^T\boldsymbol{\Phi}$ is shown in Figure \ref{fig25} (right). From this figure, large elements of the matrix are located close to its diagonal and small elements are far. Then, the basis functions $p_{i,2}$ and $p_{j,2}$ are semi-independent when the distance between their center points is sufficiently large. This property of the semi-cardinal basis functions enhances the stability of numerical results, specially for large $n$.
 \begin{figure}[ht]
\centering 
\includegraphics[scale=0.5]{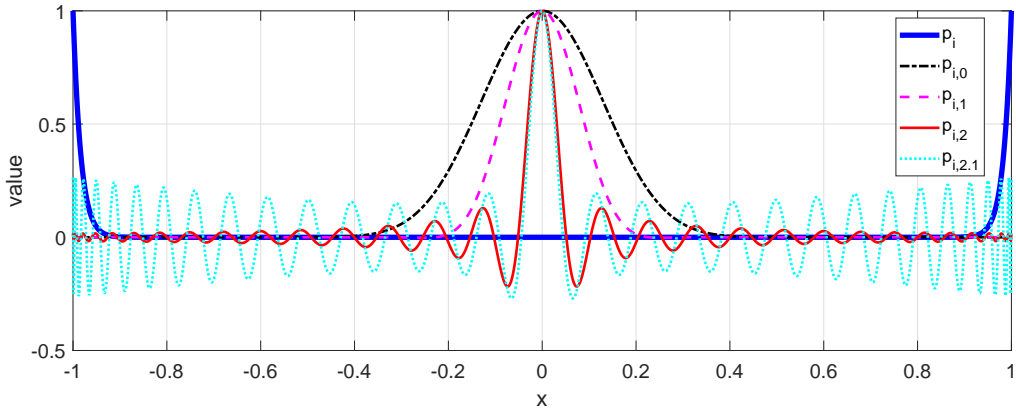} 
\caption{A cross section of polynomials $p_{i,\tau}=\prod_{k=1:n} (r_i^2 - t_k^\tau)$ for $\tau=0,1,2,2.1$ where the center point is $0$ and $n=30$. From this figure $p_{i,\tau}$ takes $1$ at the center point and vanishes at points far from center for $\tau\leq 2$.
}
\label{fig2} 
\end{figure}

\begin{figure}[ht]
\centering 
\includegraphics[scale=0.7]{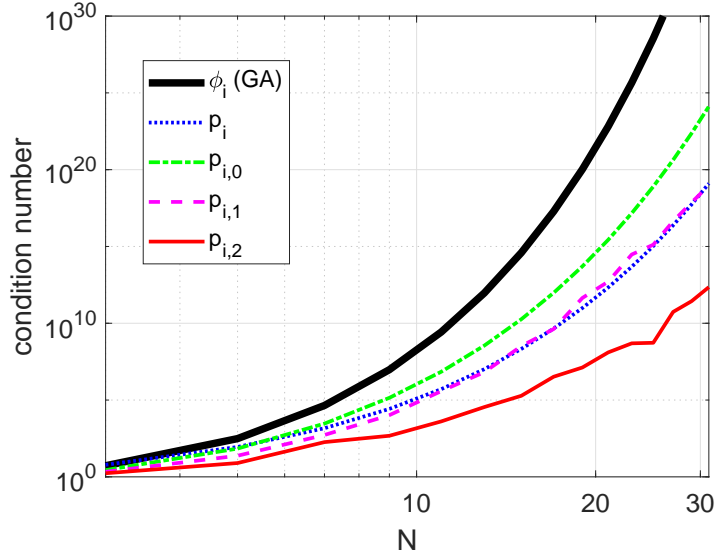} 
\caption{ Condition number of coefficient matrix $\boldsymbol{\Phi}$ for Gaussian RBF with $\epsilon=1$ and polynomial basis functions $p_i, p_{i,0}, p_{i,1}$ and $p_{i,2}$. From this figure, the condition number of $p_{i,2}$ is smaller than that of Gaussian RBF and $p_i$  for large values of $n$.
} 
\label{fig3} 
\end{figure}

\subsection{Regularized basis functions}\label{sec4_2}
Although the new basis functions $p_{i,2}$ enhance the stability, but they all are polynomials of degree $2n$. This fact reduces accuracy of Approximation (\ref{eq2.2}) for large values of $n$ because polynomials with lower degrees are not present in the set of basis functions. To overcome this drawback, basis functions with lower degrees should be replaced with some basis functions of degree $2n$. As a result, a set of regularized basis functions are obtained containing polynomials of degree $0, 2, 4, ..., 2n$ as 
\be\label{new}
q_{1,2} =1 ~~~,~~~ q_{i, 2} = \prod_{k=1:j} (r_i^2 - t_k^2) ,
\ee
for $i=h(j-1)+1, h(j-1)+2, ..., h(j)$ where $j=1, 2, ...,n$ and $t_k=\cos( (2k-1)\pi /(4j+2) )$. The proposed regularized basis functions can be presented in expanded form as below:\\
\begin{align*}
q_{1,2} &=1 ~, \\
q_{i,2} &=  (r_i^2 - t_1^2) ~,~~~~~~~~ i=2, 3, \cdots, h(1),~ t_1=\cos( \frac{1}{6} \pi ), \\
q_{i,2} &=  \prod_{k=1:2} (r_i^2 - t_k^2)  ~,~~ i=h(1)+1, h(1)+2, \cdots, h(2),~ t_k=\cos( \frac{2k-1}{10} \pi ), \\
& \vdots \\
q_{i,2} &=   \prod_{k=1:n} (r_i^2 - t_k^2)  ~,~~ i=h(n-1)+1, h(n-1)+2, \cdots,  h(n), ~ t_k=\cos( \frac{2k-1}{4n+2} \pi ) ,
\end{align*}
where the function in the first row is the base of $\mathcal{H}_0$, functions in the second row are the basis of $\mathcal{H}_1$, functions in the third row are the basis of $\mathcal{H}_2$ and so on. Note that the same strategy can be imposed for the other basis functions to find regularized basis functions. For instance, regularized monomial basis functions
\be\label{new2}
q_{1} =1 ~~~,~~~ q_{i} = r_i^{2j} ,
\ee
can be obtained from $p_i$ for $i=h(j-1)+1, h(j-1)+2, ..., h(j)$ and $j=1, 2, ...,n$.
Numerical experiments presented in Subsection \ref{sec52} show the ability of the new regularized polynomial basis functions for approximation.
Although the main objective of this paper is interpolation, but the new basis functions can also be applied for solving PDEs similar to RBFs as it is described in Subsection \ref{sec53}.

\section{Numerical experiments}\label{sec5}
The approximation power of the new linear space, $\mathcal{H}_n$, is analyzed numerically in this section. At first, the distance from smooth RBFs to linear spaces $\mathcal{H}_n$, $\mathcal{P}_{2n-1}$	 and $\mathcal{P}_{2n}$ is calculated in Subsection \ref{sec51} to find the closest linear space of polynomials to the RBFs. Then, subsections \ref{sec52} and \ref{sec53} are devoted to interpolate data and solve PDEs, respectively. In both subsections, polynomial basis functions of $\mathcal{H}_n$ are compared with Gaussian RBFs.  In comparison with Gaussian RBFs, the numerical results highlight the approximation power of the proposed polynomial basis functions.


\subsection{The distance from smooth RBFs to $\mathcal{H}_n$}\label{sec51}
The aim of this subsection is measuring the distance between smooth RBFs and $\mathcal{H}_n$. The distance from RBF $\phi_0(\mbf{x})=\phi(\mbf{x}-\mbf{x}_0)$ to $\mathcal{H}_n$ for a fixed point $\mbf{x}_0\in\Omega$ is defined as
\begin{align} \label{eq.dis1}
dis(\phi_0, \mathcal{H}_n) 
= \min_{ p\in \mathcal{H}_n}  \| \phi_0 - p \|_{2} 
=   \| \phi_0 - p_0 \|_{2} ,
\end{align}
where $p_0\in \mathcal{H}_n$ satisfies $\int_{\Omega} (\phi_0 (\mbf{x}) - p_0(\mbf{x})) p(\mbf{x}) \, d\mbf{x}=0$ for all $p\in \mathcal{H}_n$. If $p_0$ is expanded by polynomial basis functions $p_1, p_2, ..., p_{h(n)}$ of $ \mathcal{H}_n$ as $p_0=\sum_{k=1:h(n)} \alpha_k p_k$, then the unknown coefficients $\alpha_1, \alpha_2, ..., \alpha_{h(n)}$ are determined by solving system of linear equations $\mbf{B} \bsy{\alpha} = \mbf{b}$, where
\be\label{base1}
\mbf{B}[i,j]=\int_{\Omega} p_i(\bdx) p_j(\bdx) \, d \bdx ~,~~~~ 
\bsy{\alpha}[i]=\alpha_i ~,~~~~ 
\mbf{b}[i]=\int_{\Omega} p_i(\bdx) \phi_0(\bdx) \, d \bdx ~,~
\ee
for  $i, j=1, 2, ..., h(n)$. The distance between $\phi_0$ and $\mathcal{H}_n$ is calculated for some RBFs and are reported in Table \ref{tab:table2} where $\epsilon=1/2$, $\mbf{x}_0=(0,0)$, $\Omega=[-1,1]^2$ and $n \leq 7$. The studied RBFs in this subsection are 
\begin{itemize}
\item Gaussian RBF, $\phi(r)=\exp(-\epsilon^2 r^2)$, 
\item Inverse multiquadric RBF, $\phi(r)=(1+\epsilon^2 r^2)^{-1/2}$, 
\item Multiquadric RBF, $\phi(r)=(1+\epsilon^2 r^2)^{1/2}$,
\item Inverse quadric RBF, $\phi(r)=(1+\epsilon^2 r^2)^{-1}$.
\end{itemize}

 \begin{table}[ht]
  \begin{center}
    \caption{The distance from smooth RBF $\phi(\mbf{x}-\mbf{x}_0)$ to linear spaces $\mathcal{H}_{n}$, $\mathcal{P}_{2n-1}$ and $\mathcal{P}_{2n}$ where $\mbf{x}_0=(0,0)$, $\epsilon=1/2$ and $\mbf{x}\in[-1,1]^2$.}
    \label{tab:table2}
    \begin{tabular}{c|c|ccccccccccc}
      \hline
      \hline
      RBF & Linear space &        &    &   & $n$    &     &    &    &   \\
            &                   &     & 2 & 3 & 4 & 5  & 6 & 7 & \\
      \hline
               & $\mathcal{H}_{n}$        & & $4.28e{-4}$ & $1.32e{-5}$ & $3.26e{-7}$ &  $6.71e{-9}$ & $1.19e{-10}$ & $2.01e{-12}$ & \\
        GA  & $\mathcal{P}_{2n-1}$      &  & $1.06e{-2}$ & $4.20e{-4}$ & $1.25e{-5}$ &  $3.00e{-7}$ & $6.02e{-9}$ & $1.04e{-10}$ & \\
               & $\mathcal{P}_{2n}$        &  & $4.20e{-4}$ & $1.25e{-5}$ & $3.00e{-7}$ &  $6.01e{-9}$ & $1.04e{-10}$ & $2.00e{-12}$ & \\
      \hline             
              & $\mathcal{H}_{n}$       &  & $5.18e{-4}$ & $4.57e{-5}$ & $4.12e{-6}$ &  $3.79e{-7}$ & $3.54e{-8}$ & $3.33e{-9}$ & \\
       IMQ  & $\mathcal{P}_{2n-1}$      &  & $6.21e{-3}$ & $5.06e{-4}$ & $4.31e{-5}$ &  $3.78e{-6}$ & $3.38e{-7}$ & $3.08e{-8}$ & \\
               & $\mathcal{P}_{2n}$       &  & $5.06e{-4}$ & $4.31e{-5}$ & $3.78e{-6}$ &  $3.38e{-7}$ & $3.08e{-8}$ & $2.84e{-9}$ & \\
      \hline             
               & $\mathcal{H}_{n}$     &  & $1.24e{-4}$ & $7.86e{-6}$ & $5.53e{-7}$ &  $4.16e{-8}$ & $3.30e{-9}$ & $2.70e{-10}$ & \\
      MQ    & $\mathcal{P}_{2n-1}$      &  & $2.46e{-3}$ & $1.21e{-4}$ & $7.41e{-6}$ &  $5.07e{-7}$ & $3.72e{-8}$ & $2.87e{-9}$ & \\
              & $\mathcal{P}_{2n}$     &  & $1.21e{-4}$ & $7.42e{-6}$ & $5.07e{-7}$ &  $3.72e{-8}$ & $2.87e{-9}$ & $2.30e{-10}$ & \\
      \hline             
               & $\mathcal{H}_{n}$      &  & $1.52e{-3}$ & $1.53e{-4}$ & $1.53e{-5}$ &  $1.53e{-6}$ & $1.54e{-7}$ & $1.55e{-8}$ & \\
       IQ    & $\mathcal{P}_{2n-1}$      &  & $1.52e{-2}$ & $1.48e{-3}$ & $1.43e{-4}$ &  $1.40e{-5}$ & $1.37e{-6}$ & $1.34e{-7}$ & \\
               & $\mathcal{P}_{2n}$      &  & $1.48e{-3}$ & $1.43e{-4}$ & $1.40e{-5}$ &  $1.37e{-6}$ & $1.34e{-7}$ & $1.31e{-8}$ & \\
      \hline             
      \hline
    \end{tabular}
  \end{center}
\end{table}
From this table the distance vanishes exponentially when $n$ increases and it is smaller than $10^{-8}$ for $n > 7$. 
Also, results of Table  \ref{tab:table2} reveal that the Gaussian RBF is more closer to the new linear space than the other RBFs. Moreover, the distance from the RBFs to linear spaces  $\mathcal{P}_{2n-1}$ and $\mathcal{P}_{2n}$ also is calculated and is presented in Table \ref{tab:table2}. From numerical results of  Table  \ref{tab:table2}, we have 
\[dis(\phi_0, \mathcal{P}_{2n}) \simeq dis(\phi_0, \mathcal{H}_n) < dis(\phi_0, \mathcal{P}_{2n-1}),
\]
 while $\mathcal{H}_n \subset \mathcal{P}_{2n}$ and dimension of $\mathcal{H}_n$ is significantly smaller than the dimensions of $\mathcal{P}_{2n-1}$ and $\mathcal{P}_{2n}$ for large values of $n$. This fact reveals that $\mathcal{H}_n$ is the smallest subspace of $\mathcal{P}_{2n}$ which the distance from the smooth RBFs to it tends to zero as fast as possible when $n$ increases.

\subsection{The use of $\mathcal{H}_n$ for interpolation}\label{sec52}
In this subsection, the approximation power of $\mathcal{H}_n$ is compared with that of RBFs. Four kinds of functions,  $p_i$, $p_{i,2}$, $q_i$ and $q_{i,2}$, are considered as basis of $\mathcal{H}_n$ and Gaussian function $\exp(-\epsilon r_i^2)$ is assumed as radial basis function when $\epsilon$ varies from $10^{-2}$ to $10^2$. Smooth function 
\be\label{sin}
u(\mbf{x})= \sin(x+y) ,
\ee
is interpolated by the polynomial and radial basis functions defined in $\Omega=[0,1]^2$. 
Center nodal points are considered in both regular and irregular forms in $\Omega$ as is shown in Figure \ref{fig5}. As it is reported in Equation (\ref{Dn}) dimension of $\mathcal{H}_n$ is $h(n)=(n+1)^2$ for two dimensional problems and we need at least $N=h(n)$ center points to construct the linear space.  
Root mean square error (RMSE) of interpolation is calculated as
\be\label{err}
error= \frac{1}{\sqrt{ N }} \sqrt{ \sum_{i=1:N} ( u(\mbf{y}_i)-\bar{u}(\mbf{y}_i) )^2} ,
\ee
where $u$ and $\bar{u}$ are analytical and numerical solutions, respectively, and $\mbf{y}_i\in[0,1]^2$ is selected in $[0,1]^2$, randomly for $i=1, 2, ..., N$. The error of interpolation is reported in Figure \ref{fig5} for $N=49, 121$ and $441$ number of nodal points. From this figure, accuracy of the polynomial basis functions is less than the best ones of the RBFs for small values of $N$, but the accuracy of the regularized basis functions, $q_i$ and $q_{i,2}$, increases significantly when $N$ gets large. The stability of basis functions $q_{i,2}$ leads to more accurate results for larger values of $N$ as it is shown in third column of  Figure \ref{fig5}.    
This shows polynomial basis functions $q_{i,2}$ are appropriate alternative for the RBFs to approximate scattered data. Note that, since polynomials are ill-condition at the regular points, the error of Halton points, reported in the second row of Figure \ref{fig5}, is significantly smaller than that of regular points reported in the first row of Figure \ref{fig5} for large values of $N$. 
Also for small values of shape parameter ($\epsilon^2 \leq 0.1$) the distance from Gaussian RBF to $\mathcal{H}_{11}$ is less than the machine epsilon, $\delta_{mach.} \simeq 2.2 \times 10^{-16}$, and one can suppose $S \simeq \mathcal{H}_{11}$, numerically. So the error of interpolation with RBFs should be less than or equal to that of $q_{i,2}$ when $\epsilon$ is small. But the ill-conditioning of the RBFs does not let it. 

\begin{figure}[ht]
\centering 
\includegraphics[width=0.45\textwidth]{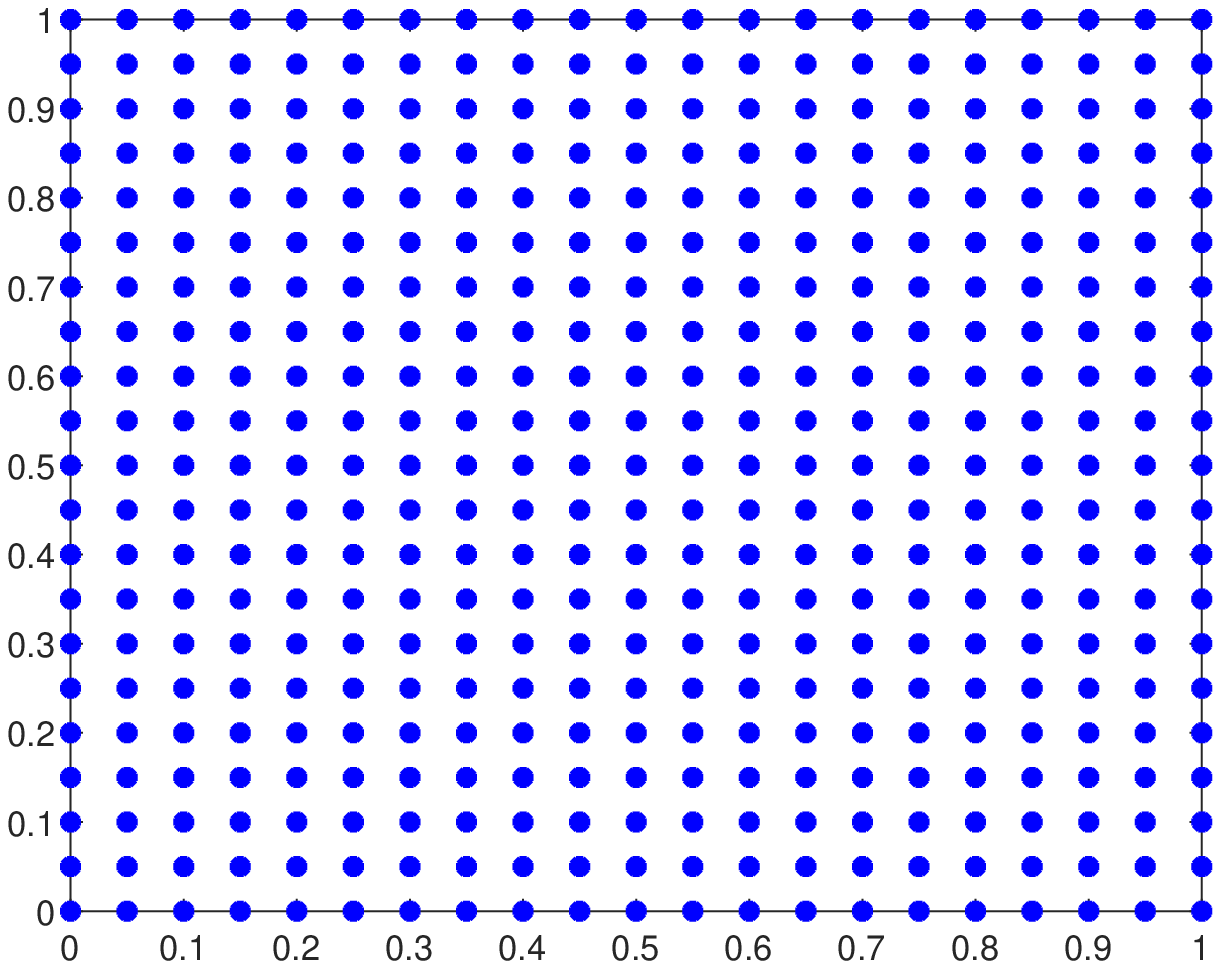}
\includegraphics[width=0.45\textwidth]{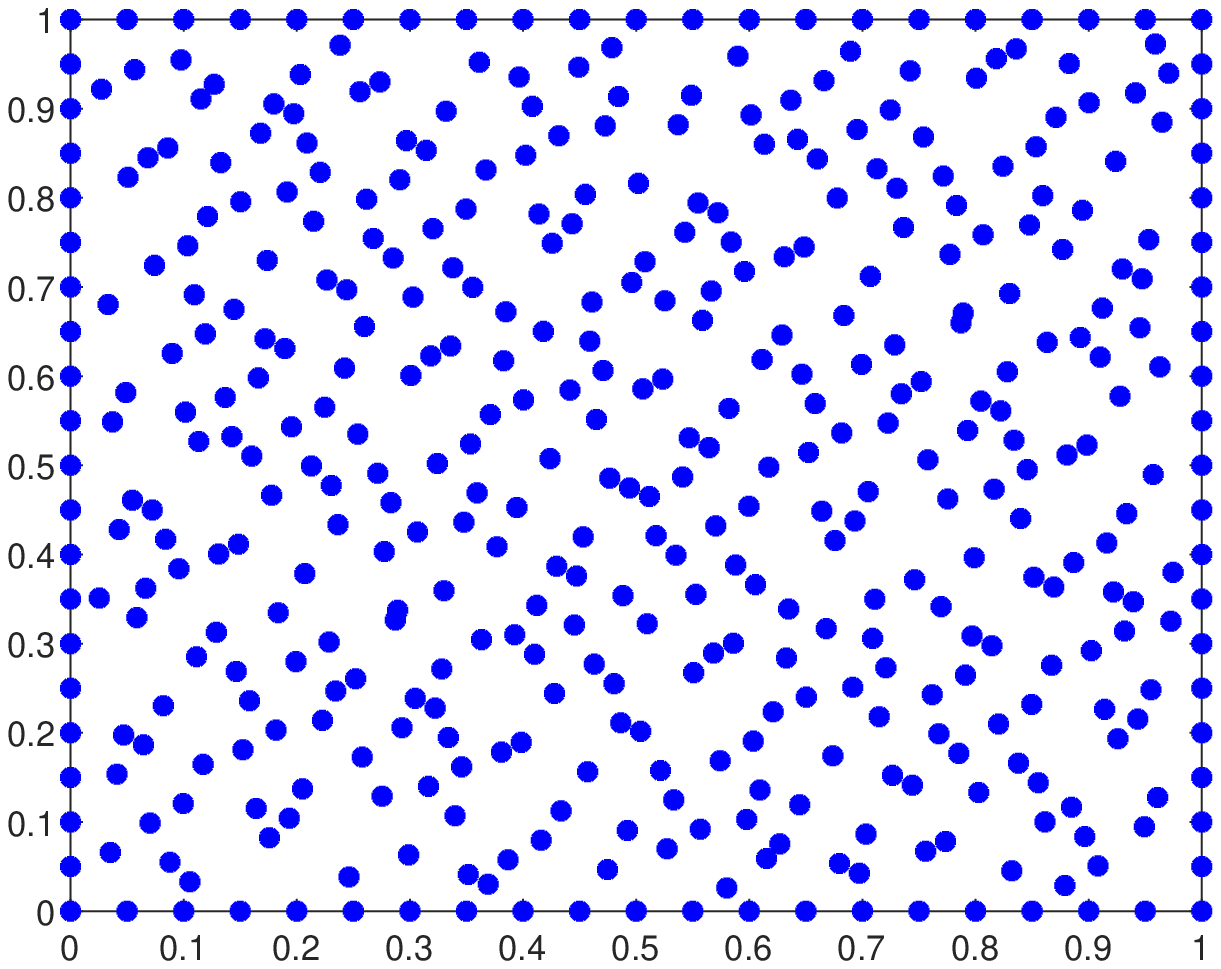} 
\caption{Left: regularly spaced nodal points in square computational domain. 
Right: Halton points considered as irregular distributed nodal points. }
\label{fig6} 
\end{figure}

\begin{figure}[ht]
\centering 
\includegraphics[width=0.3\textwidth]{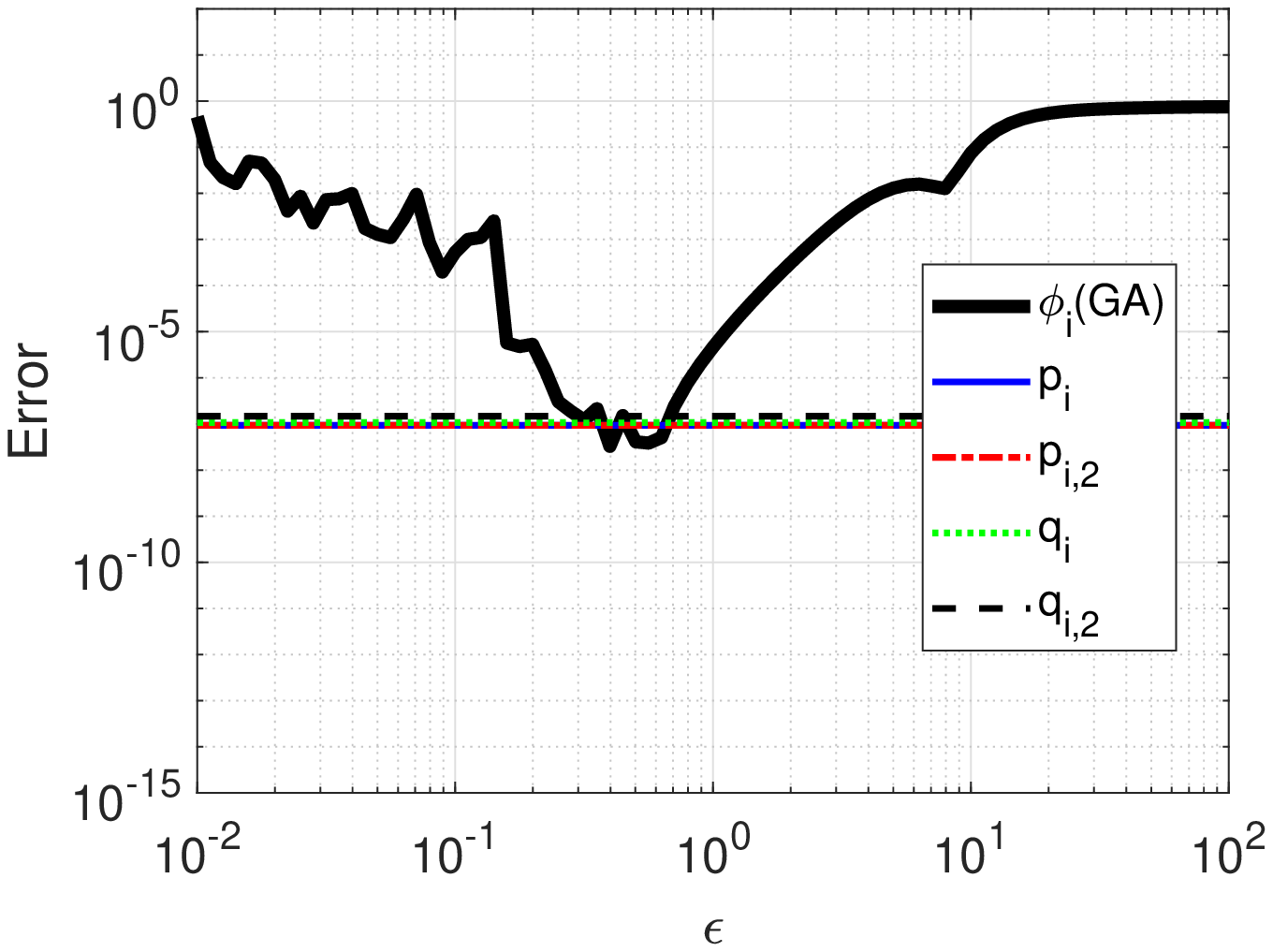}
\includegraphics[width=0.3\textwidth]{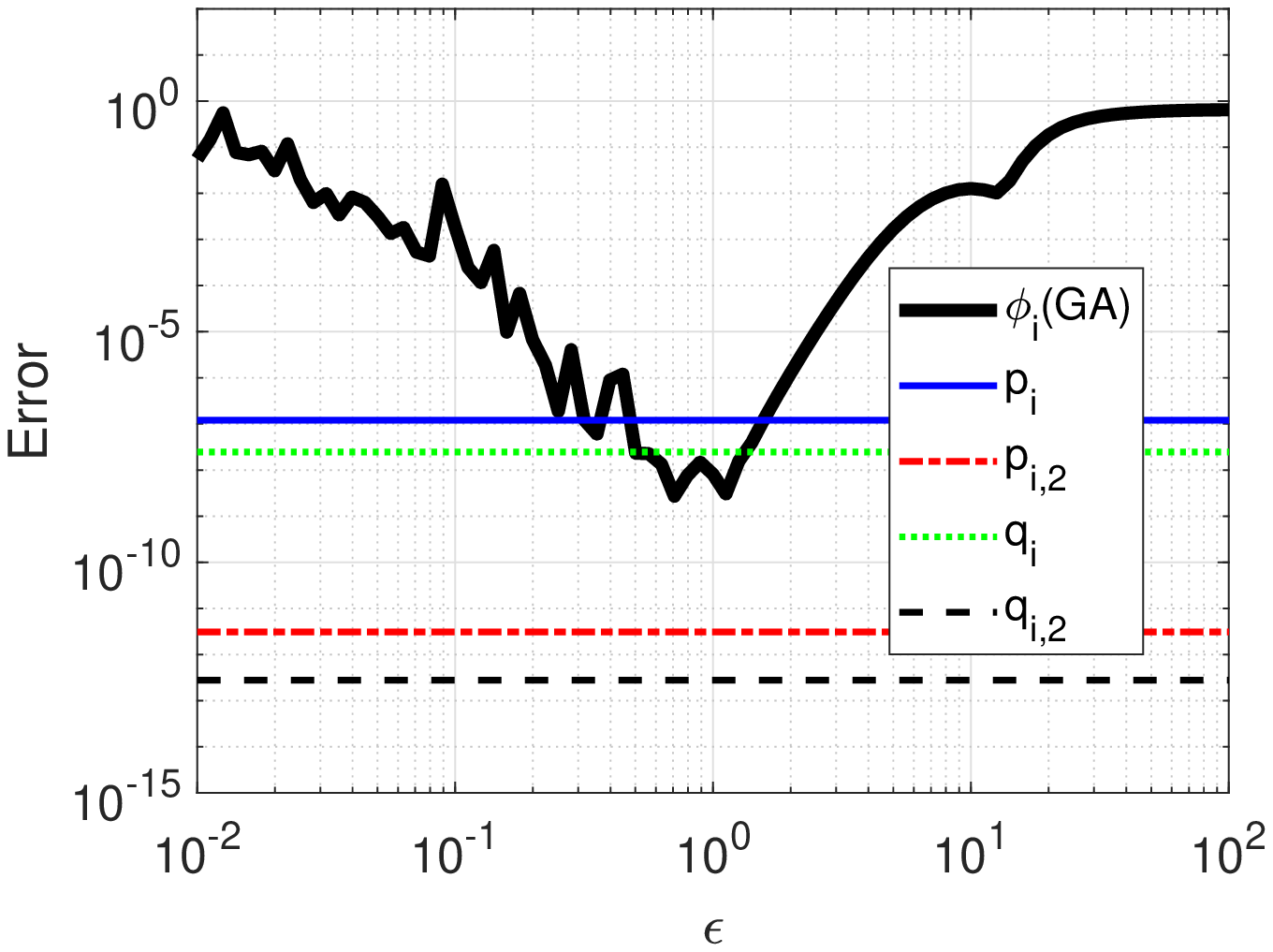}
\includegraphics[width=0.3\textwidth]{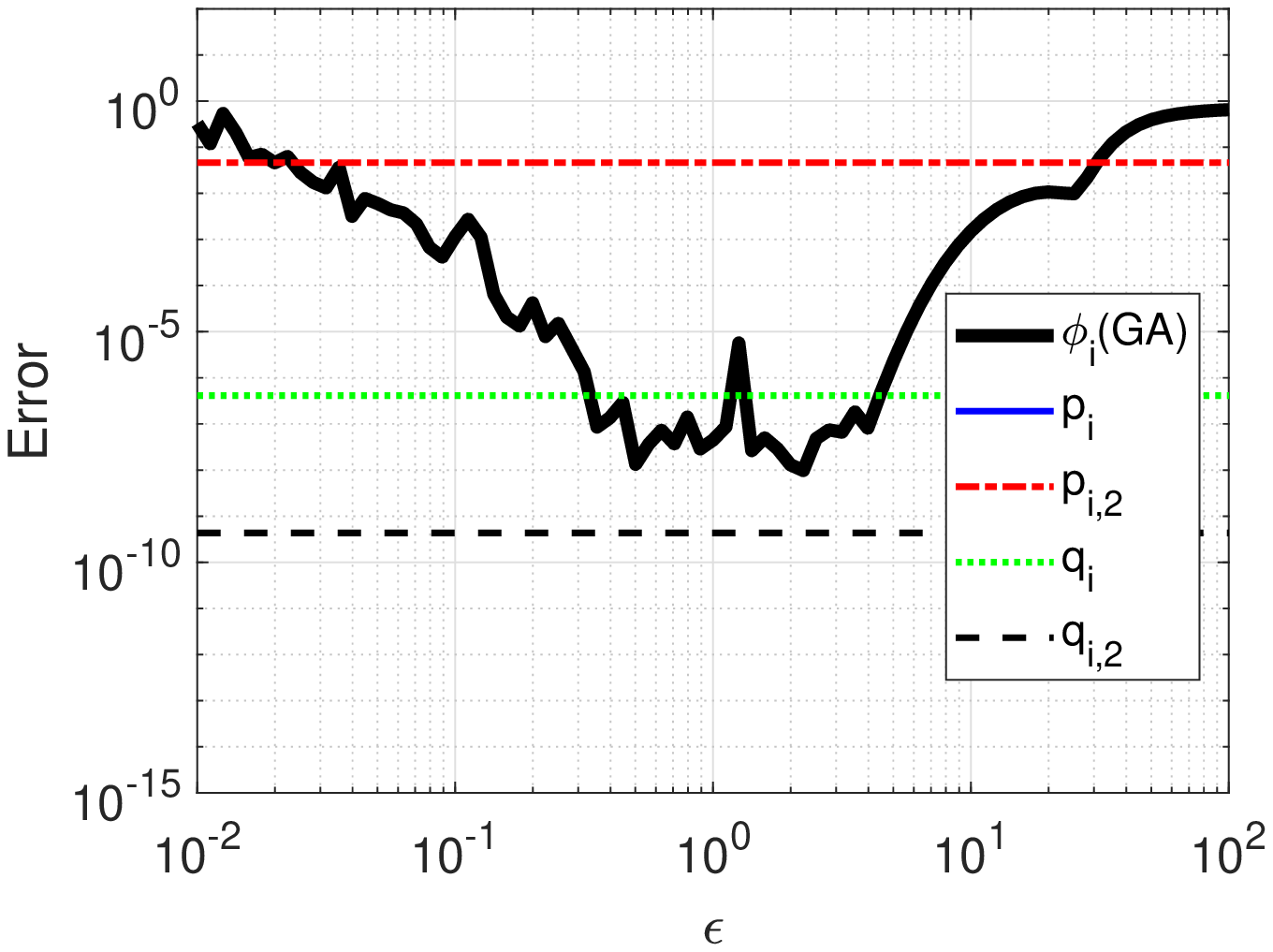}
\includegraphics[width=0.3\textwidth]{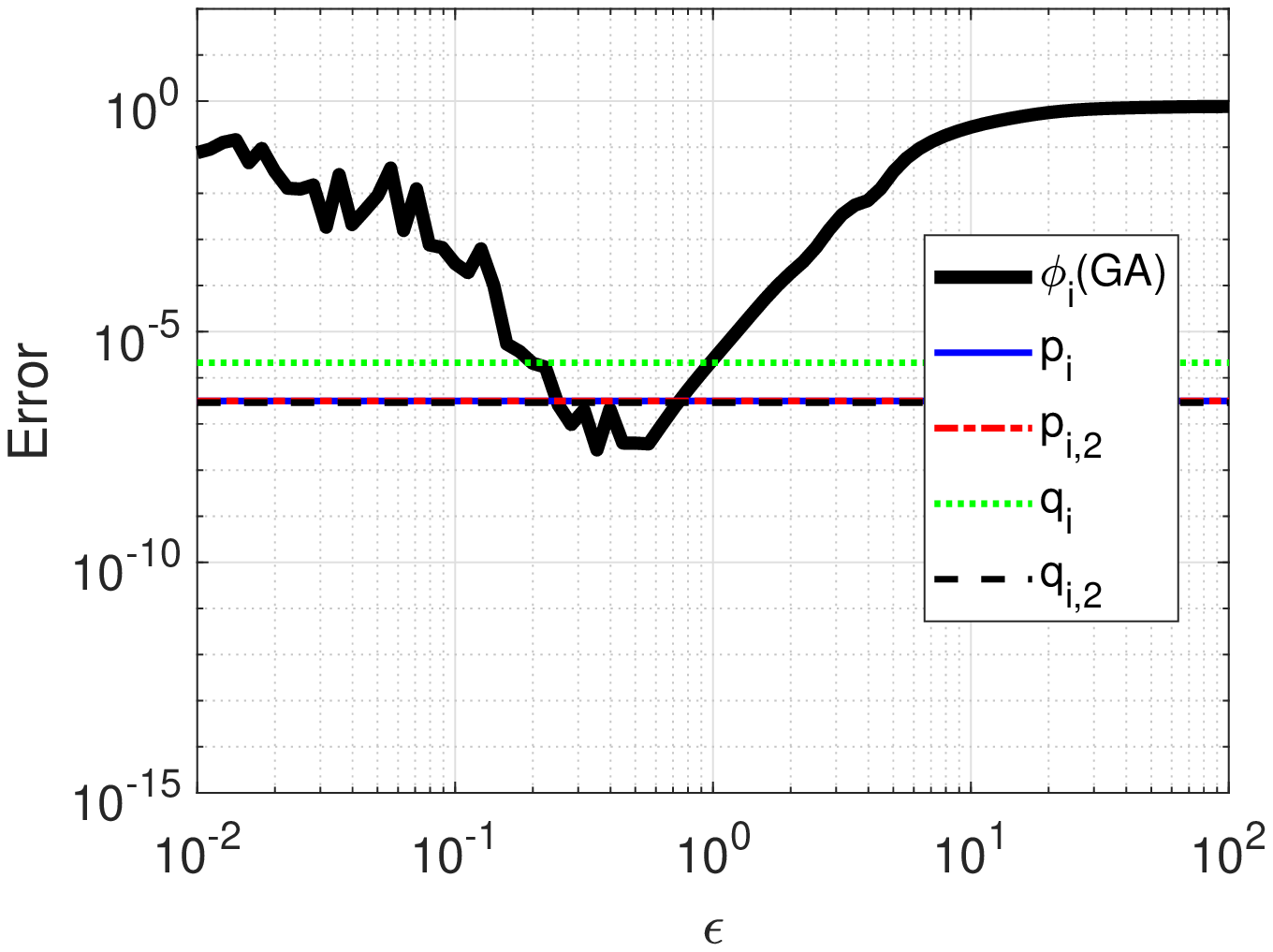}  
\includegraphics[width=0.3\textwidth]{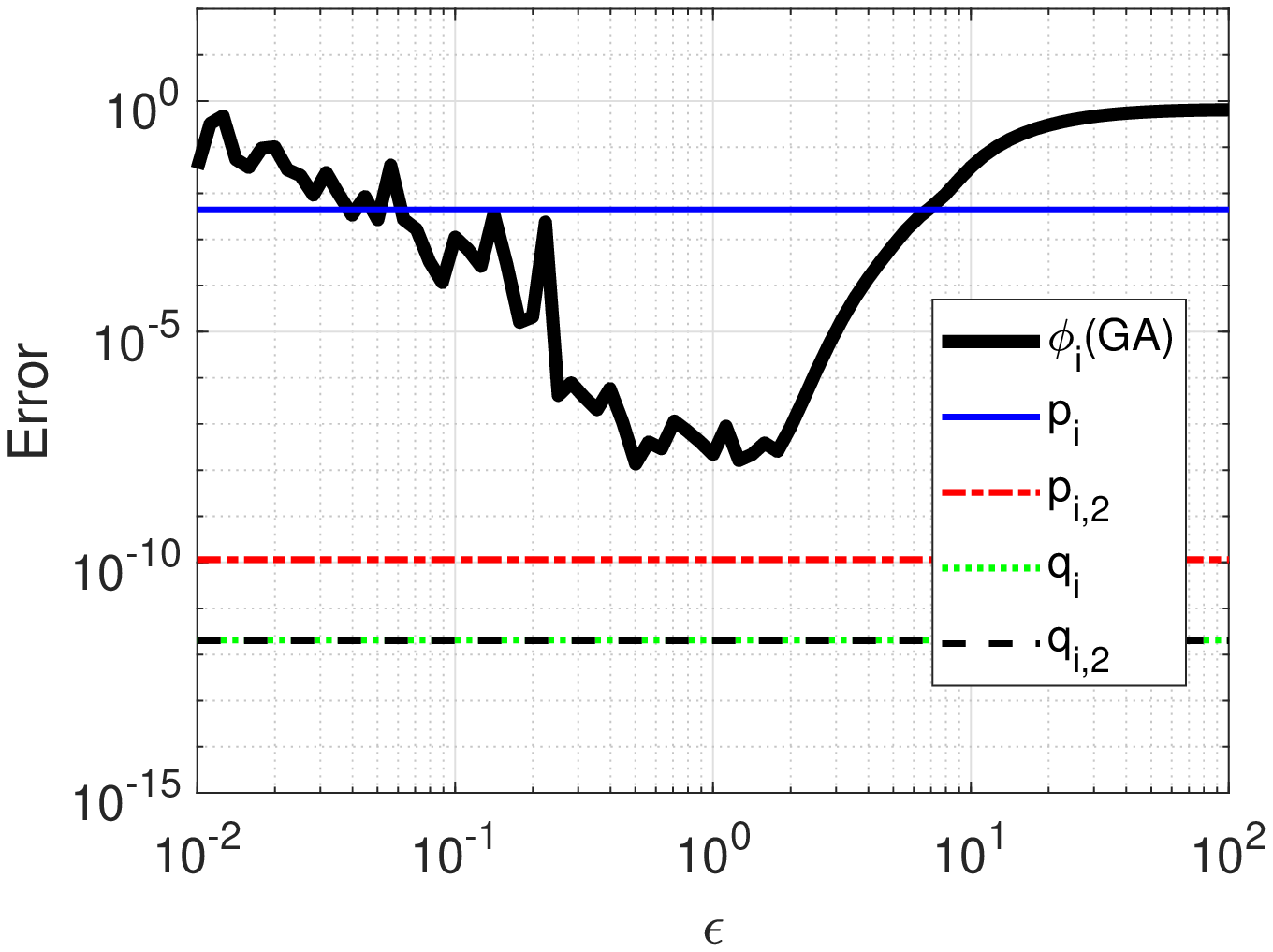}
\includegraphics[width=0.3\textwidth]{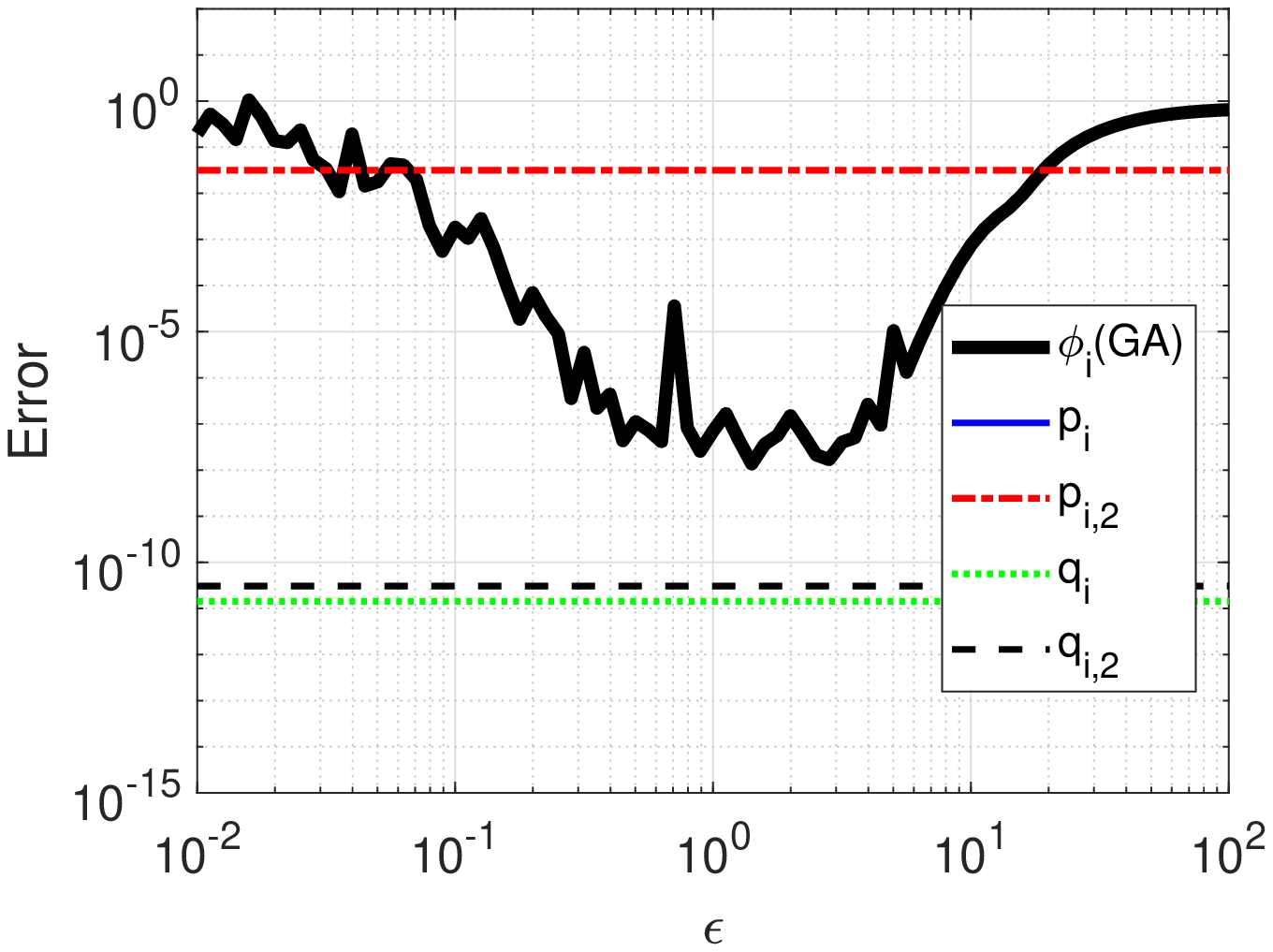} 
\caption{The RMSE for polynomial basis functions $p_i$, $p_{i,2}$, $q_i$, $q_{i,2}$ and Gaussian RBF  versus shape parameter $\epsilon$. The interpolation points are regular and irregular for the top and bottom rows, respectively. Also the first, second and third columns are respect to $N=9, 121$ and $441$, respectively.} 
\label{fig5} 
\end{figure} 

The same investigation is done for the star-shaped domain illustrated in Figure \ref{fig66} (left). Function (\ref{sin}) is interpolated by the polynomial and radial basis functions at $N=121$ scattered points and RMSE (\ref{err}) is calculated for some internal points selected in the domain, randomly. The error is shown in Figure \ref{fig66} (right). The figure reveals that accuracy of regularized polynomial basis functions $q_{i}$ and $q_{i,2}$ is significantly higher than that of the other polynomial and radial basis functions. 

\begin{figure}[ht]
\centering 
\includegraphics[width=0.45\textwidth]{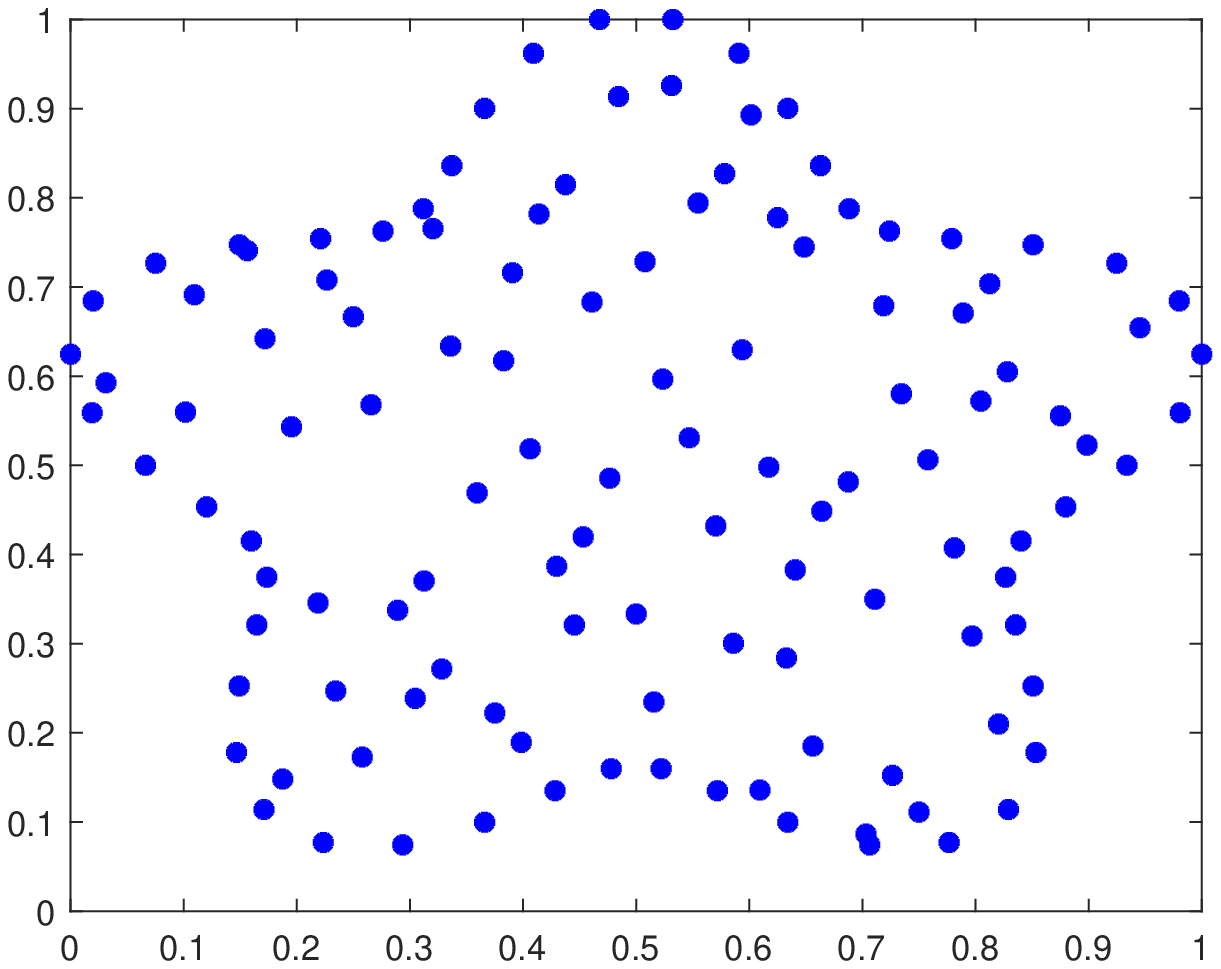}
\includegraphics[width=0.45\textwidth]{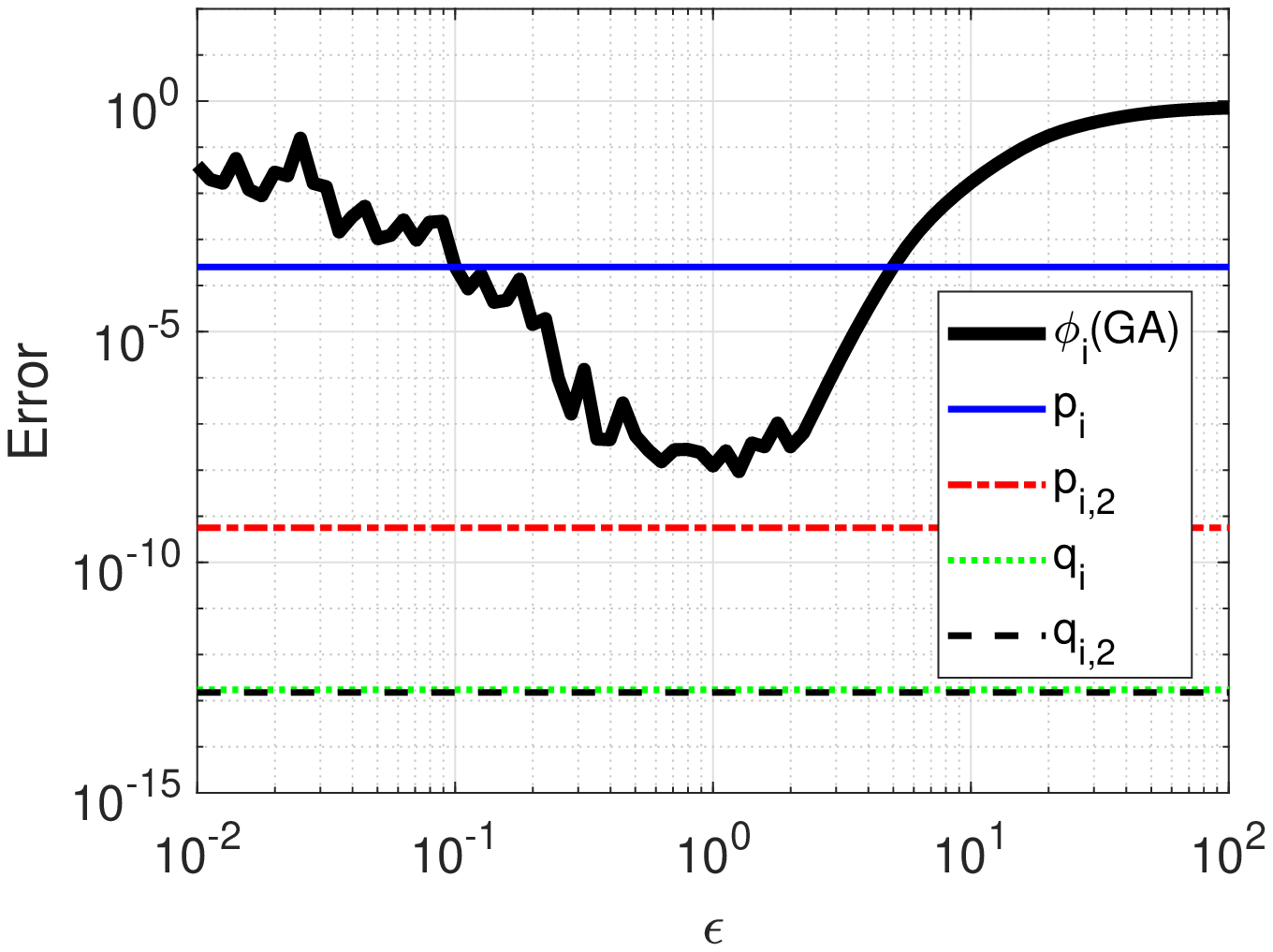}
\caption{The interpolation points for the star-shaped domain (left) and the RMSE of the interpolation (right) for the polynomial and radial basis functions versus shape parameter $\epsilon$. }
\label{fig66} 
\end{figure} 

The polynomial basis functions also are applied to interpolate three dimensional data points demonstrated in Figure \ref{fig666} (left). In this case study, $N=1331$ semi-regular center points are selected in $[0,1]^3$ as interpolation points and the error of interpolation is calculated by RMSE formulation (\ref{err}) when exact solution is
\[
u(x,y,z)=\exp(x+y+z).
\]
The error is shown in Figure \ref{fig666} (right) and it reveals that regularized basis functions $q_i$ and $q_{i,2}$ are significantly more accurate than that of the other polynomial and radial basis functions. 

\begin{figure}[ht]
\centering 
\includegraphics[width=0.45\textwidth]{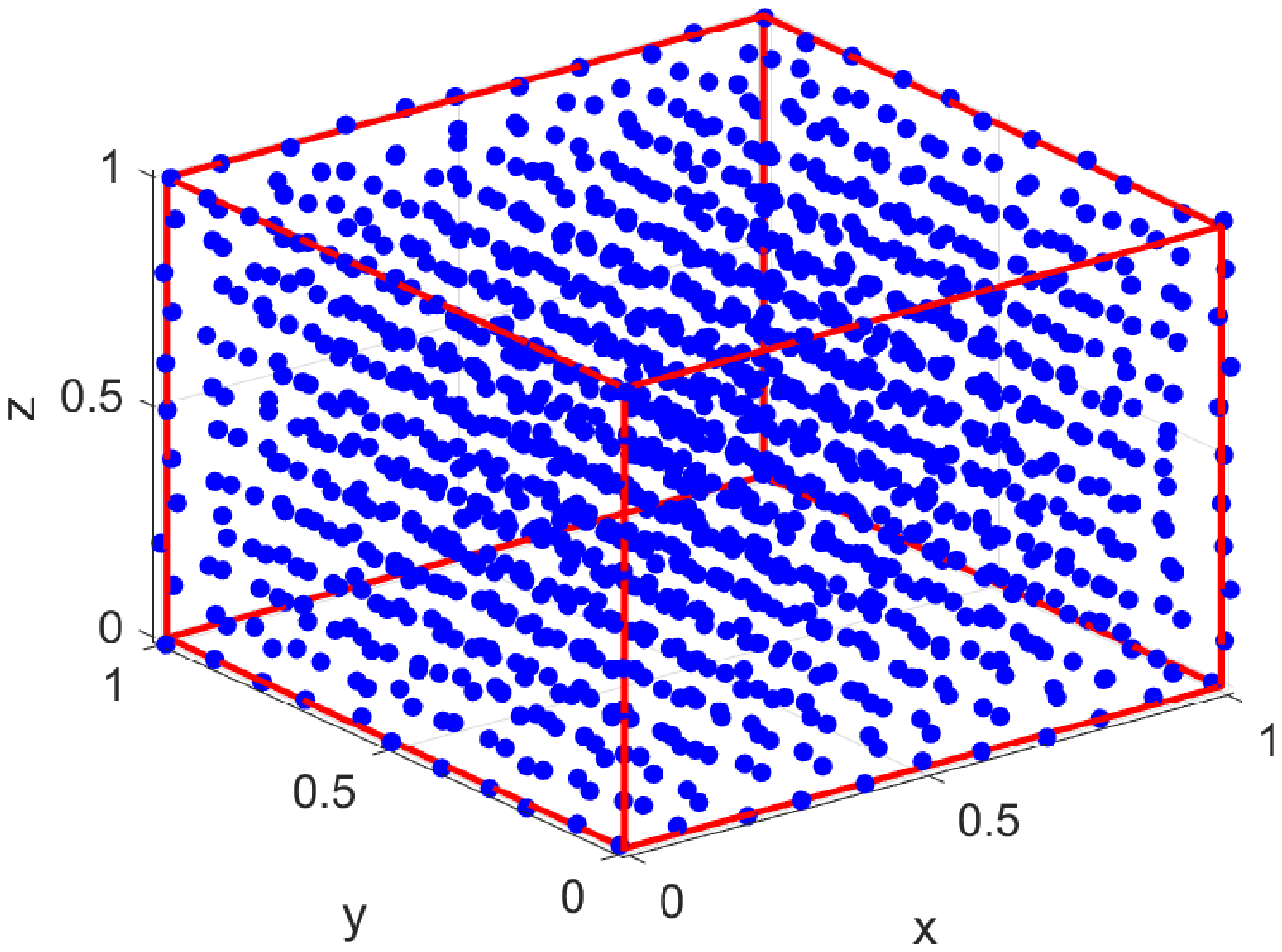}
\includegraphics[width=0.45\textwidth]{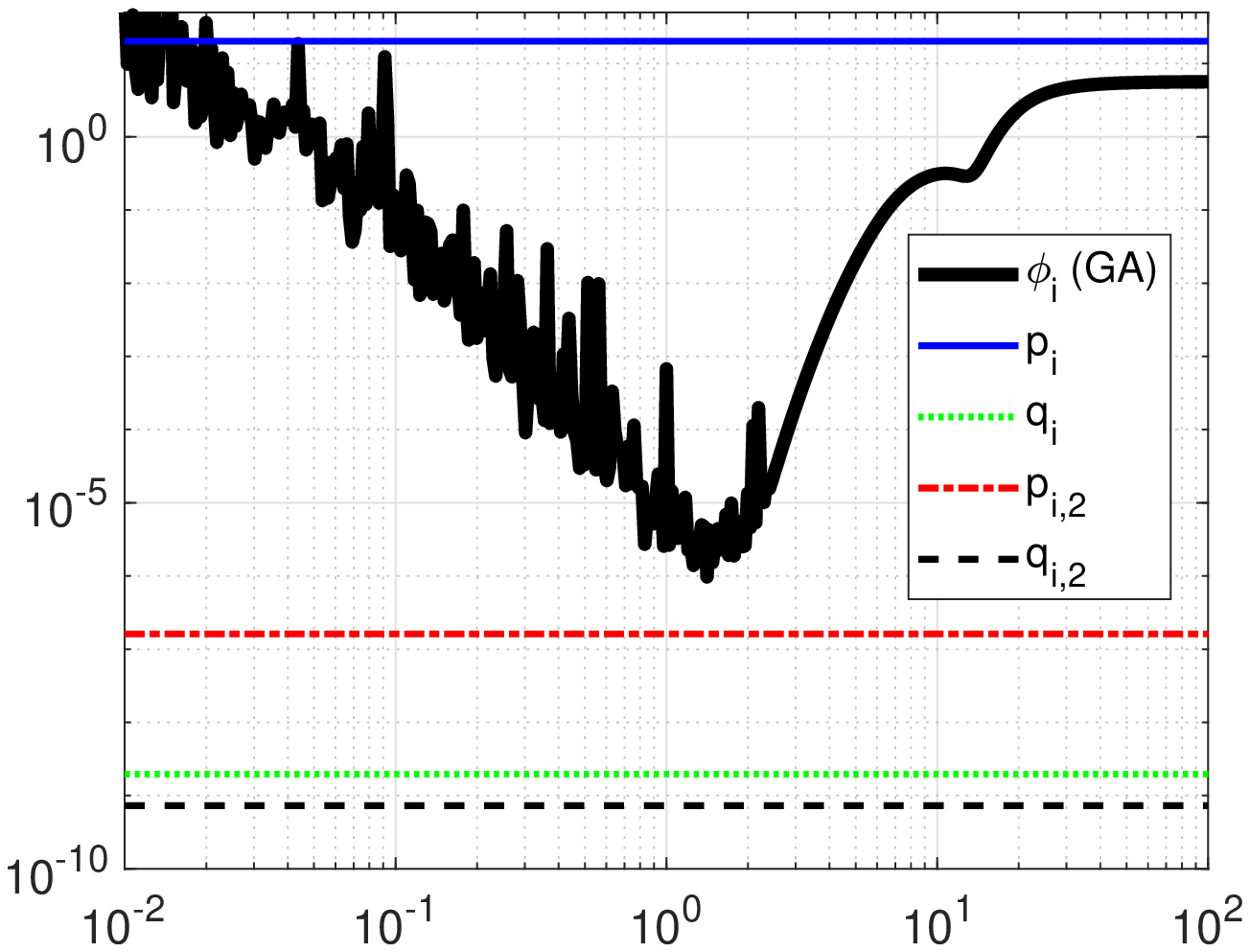}
\caption{The interpolation points for the cube domain (left) and the RMSE of the interpolation (right) for the polynomial and radial basis functions  versus shape parameter $\epsilon$.}
\label{fig666} 
\end{figure}

\subsection{The use of $\mathcal{H}_n$ for solving PDEs}\label{sec53}
To see the efficiency of polynomial basis functions proposed in this paper for solving PDEs, consider Poisson equation
\[
\Delta u (\mbf{x}) = f(\mbf{x}) ,
\]
for $\mbf{x}\in [0,1]^2$. The problem is considered to be solved with a collocation method. The exact solution and right hand side functions are assumed as $u(x,y)=\sin(x+y)$ and $f(x,y)=-2 \sin(x+y)$, respectively, and $N=441$ center points are selected in $[0,1]^2$ regularly and irregularly as is shown in Figure \ref{fig6}. We are looking for function $\bar{u}$ satisfying the PDE and Dirichlet boundary condition at internal and boundary collocation points, respectively. Then the numerical solution $\bar{u}(\mbf{x})=\sum_{k=1:N} \lambda_k \phi(r_k(\mbf{x}))$ should satisfy equations:
\begin{align*}
\Delta \bar{u}(\mbf{x}_i)  &= f(\mbf{x}_i) ,~~~ \mbf{x}_i\in\Omega , \\ 
 \bar{u}(\mbf{x}_i) &= u(\mbf{x}_i) ,~~~ \mbf{x}_i\in\Gamma ,
\end{align*}
where $\Gamma$ is the boundary of $\Omega=[0,1]^2$. Unknown coefficient vector $\bsy{\lambda}=[\lambda_1, \lambda_2, ..., \lambda_N]$ is determined by solving system of linear equations $\mbf{C} \bsy{\lambda} = \mbf{b}$ where $\mbf{C}$ and $\mbf{b}$ are evaluated as
$
\mbf{C}[i,j]= \Delta \phi(r_i(\bdx_j)) ~, 
\mbf{b}[i]=f(\bdx_j)  ~
$
when $\mbf{x}_i\in\Omega$ and
$
\mbf{C}[i,j]=  \phi(r_i(\bdx_j)) ~, 
\mbf{b}[i]=u(\bdx_j)  ~
$
when $\mbf{x}_i\in\Gamma$.
The accuracy of results versus shape parameter, $\epsilon$, is presented in Figure \ref{fig7}. The results of this figure are obtained by considering all basis functions  $p_i, p_{i,2}, q_i, q_{i,2}$ and Gaussian RBF.

RMSE of the solution, $\bar{u}$, is calculated by Equation (\ref{err}) at the center points and is reported in Figure \ref{fig7} when the basis functions are polynomials $p_i, p_{i,2}, q_i, q_{i,2}$ and Gaussian RBF with several shape parameters. From the figure, the error of regularized basis functions $q_i$ and $q_{i,2}$ is smaller than $10^{-7}$ and $10^{-11}$ for the regular and Halton points, respectively. The ill-conditioning of the polynomial basis functions in regular points reduces the accuracy while this drawback is removed for the irregular points. From Figure \ref{fig7}, the error of the best polynomials is smaller than the error of the best RBFs specially for Halton points.  

\begin{figure}[ht]
    \centering 
    \includegraphics[width=0.45\textwidth]{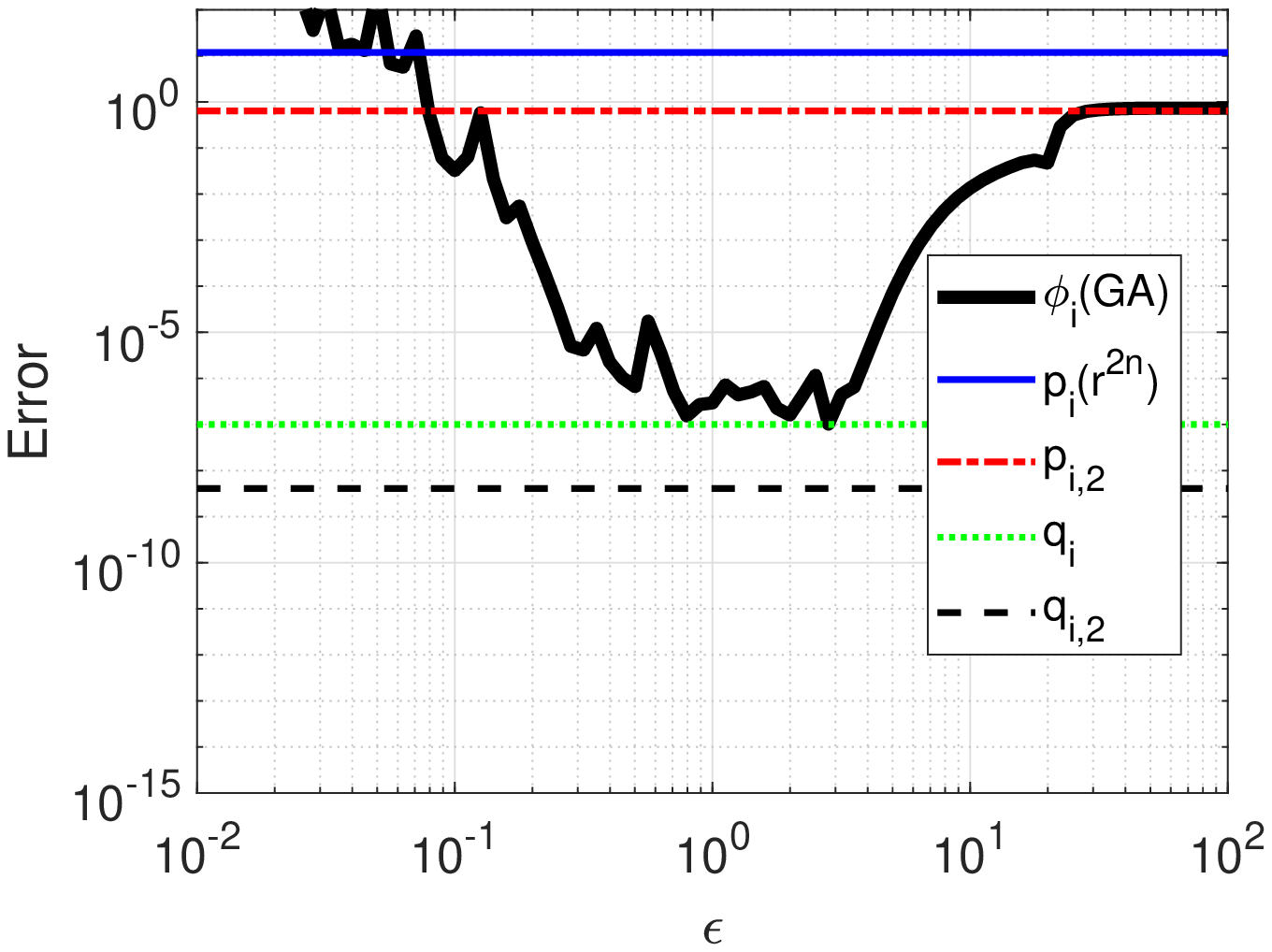}
     \includegraphics[width=0.45\textwidth]{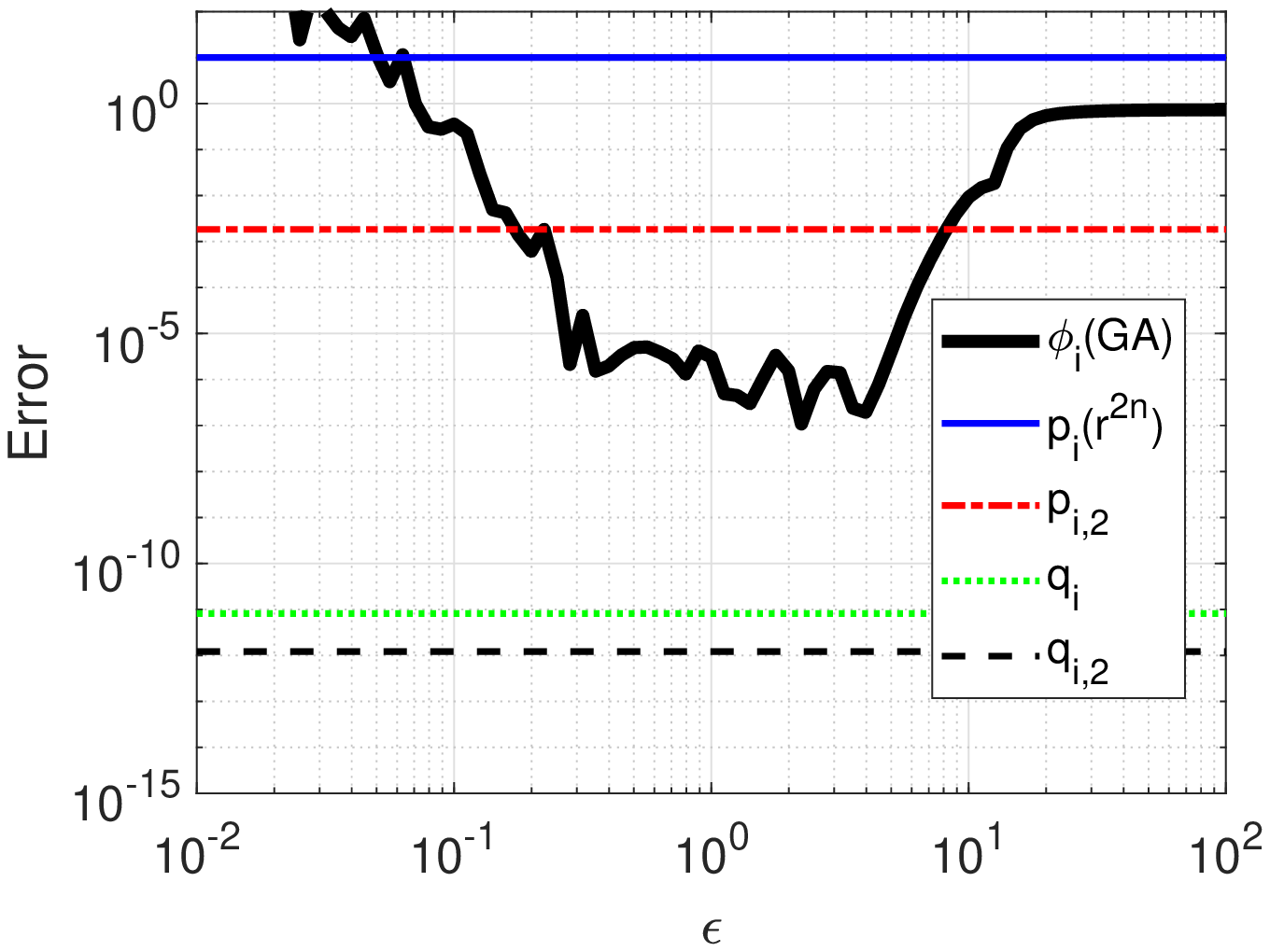}
    \caption{The RMSE for polynomial basis functions $p_i$, $p_{i,2}$, $q_i$, $q_{i,2}$ and Gaussian RBF with shape parameter $\epsilon$ to solve the PDE. The center points are regular and irregular for the left and right figures. } 
\label{fig7} 
\end{figure}

\section{Conclusion and final remarks} \label{sec6}
Motivated by Taylor’s expansion of infinitely smooth radial basis functions (RBFs), a new linear space of polynomials, denoted by $\mathcal{H}_n$, was proposed. For the purpose of analysis, a set of basis functions was obtained for the new linear space and its dimension was specified. Then, some important relations between $\mathcal{H}_n$ and the linear space of polynomials of degree $n$, $\mathcal{P}_n$ were found, analytically. Besides, some more appropriate basis functions where proposed for $\mathcal{H}_n$ to enhance its efficiency. 
The fast convergence of the proposed polynomial space to RBFs suggests the use of polynomial basis functions of the new linear space instead of smooth RBFs to interpolate scattered data and solving partial differential equations (PDEs). Several numerical examples presented in this paper verifies this fact. For the future works, the authors suggest finding new orthogonal or compact support polynomial basis functions for $\mathcal{H}_n$. Future works also can be involved investigating how this space can be helpful to enrich numerical stability and accuracy of smooth RBFs and finding a new optimal shape parameter for smooth RBFs. 




\bibliographystyle{plain}
\bibliography{mybibfile}



\end{document}